\documentclass[twoside,a4paper,reqno,11pt]{amsart} 
\usepackage[top=28mm,right=30mm,bottom=28mm,left=30mm]{geometry}

\usepackage{amsfonts, amsmath, amssymb, mathrsfs, bm, latexsym, stmaryrd, array, hyperref, mathtools, bold-extra, xcolor}

\renewcommand{\a}{\alpha}
\renewcommand{\b}{\beta}

\newcommand{\e}{\varepsilon}
\renewcommand{\l}{\lambda}

\renewcommand{\O}{\Omega}

\newcommand{\normeq}{\trianglelefteqslant}

\newcommand{\C}{\mathcal{C}}

\newcommand{\la}{\langle}
\newcommand{\ra}{\rangle}

\renewcommand{\to}{\rightarrow}

\newcommand{\leqs}{\leqslant}
\newcommand{\geqs}{\geqslant}

\newcommand{\vs}{\vspace{2mm}}

\newcommand{\what}{\widehat} 
\newcommand{\Aut}{\operatorname{Aut}}

\newcommand{\GL}{\operatorname{GL}}
\newcommand{\PGL}{\operatorname{PGL}}
\newcommand{\SL}{\operatorname{SL}}

\newcommand{\PSigmaL}{\operatorname{P\Sigma L}}

\newcommand{\PSp}{\operatorname{PSp}}
\newcommand{\Sp}{\operatorname{Sp}}

\newcommand{\POmega}{\operatorname{P\Omega}}
\newcommand{\LL}{\operatorname{L}}
\newcommand{\UU}{\operatorname{U}}

\newcommand{\GU}{\operatorname{GU}}

\makeatletter
\newcommand{\imod}[1]{\allowbreak\mkern4mu({\operator@font mod}\,\,#1)}
\makeatother

\newtheorem{theorem}{Theorem} 
\newtheorem*{conj*}{Conjecture}

\newtheorem{thm}{Theorem}[section] 
\newtheorem{prop}[thm]{Proposition} 
\newtheorem{lem}[thm]{Lemma}
\newtheorem{cor}[thm]{Corollary}

\theoremstyle{definition}
\newtheorem{rem}[thm]{Remark}
\newtheorem{remk}{Remark}

\newtheorem{ex}[thm]{Example}

\begin{document}

\author{Timothy C. Burness}
\address{T.C. Burness, School of Mathematics, University of Bristol, Bristol BS8 1UG, UK}
\email{t.burness@bristol.ac.uk}

\author{Hong Yi Huang}
\address{H.Y. Huang, School of Mathematics, University of Bristol, Bristol BS8 1UG, UK}
\email{hy.huang@bristol.ac.uk}

\title[Saxl graphs of primitive groups]{On the Saxl graphs of primitive groups \\ with soluble stabilisers} 

\begin{abstract}
Let $G$ be a transitive permutation group on a finite set $\O$ and recall that a base for $G$ is a subset of $\O$ with trivial pointwise stabiliser. The base size of $G$, denoted $b(G)$, is the minimal size of a base. If $b(G)=2$ then we can study the 
Saxl graph $\Sigma(G)$ of $G$, which has vertex set $\O$ and two vertices are adjacent if and only if they form a base. This is a vertex-transitive graph, which is conjectured to be connected with diameter at most $2$ when $G$ is primitive. In this paper, we combine probabilistic and computational methods to prove a strong form of this conjecture for all almost simple primitive groups with soluble point stabilisers. In this setting, we also establish best possible lower bounds on the clique and independence numbers of $\Sigma(G)$ and we determine the groups with a unique regular suborbit, which can be interpreted in terms of the valency of $\Sigma(G)$.
\end{abstract}

\date{\today}

\maketitle

\section{Introduction}\label{s:intro}

Let $G \leqs {\rm Sym}(\O)$ be a transitive permutation group on a finite set $\O$ with point stabiliser $H$ and recall that a \emph{base} for $G$ is a subset of $\O$ with trivial pointwise stabiliser. In turn, the \emph{base size} of $G$, denoted $b(G)$, is the minimal size of a base. This is a classical concept in permutation group theory, which arises naturally in a wide range of contexts. There is a long history of studying bases, which stretches all the way back to the nineteenth century, and there are many applications and connections to other areas of algebra and combinatorics. We refer the reader to the survey articles \cite{BC,LSh3} and \cite[Section 5]{Bur181} for further details. 

Historically, there has been a focus on studying bases for finite primitive permutation groups and there have been several major advances in this direction in recent years. For example, a number of highly influential conjectures due to Babai, Cameron, Kantor and Pyber from the 1990s have been resolved, which in turn has opened up new directions of research. For example, Seress \cite{Seress} proved that $b(G) \leqs 4$ if $G$ is a finite soluble primitive group, which established a strong form of Pyber's base size conjecture in this setting (in \cite{Pyber}, Pyber conjectured that there exists an absolute constant $c$ such that $b(G)$ is at most $c \log_{n}|G|$ for every primitive group $G$ of degree $n$ and the proof was finally completed in \cite{DHM}). Seress' theorem has recently been extended in \cite{Burness2020base}, where the main result shows that $b(G) \leqs 5$ for every finite primitive group with soluble point stabilisers (it is worth noting that the bounds in \cite{Burness2020base} and \cite{Seress} are best possible). Moreover, the exact base size is computed in \cite{Burness2020base} for all almost simple primitive groups with soluble point stabilisers (recall that $G$ is \emph{almost simple} if there exists a non-abelian simple group $G_0$ such that $G_0 \normeq G \leqs {\rm Aut}(G_0)$; here $G_0$ is the \emph{socle} of $G$).

Further motivation stems from an ambitious project initiated by Jan Saxl in the 1990s. Here the main aim is to classify all the base-two finite primitive permutation groups, where a group $G$ is said to be \emph{base-two} if $b(G)=2$. These groups arise naturally in many applications of bases to other problems. For example, see \cite{ExtremelyPrimitive} on the classification of extremely primitive groups and \cite{BurnessHarper} for applications concerning the $2$-generation properties of almost simple groups. Although a complete classification of the base-two primitive groups remains out of reach, there has been some significant progress. For instance, we refer the reader to \cite{faw2,faw1} for work of Fawcett on diagonal-type groups and twisted wreath products, respectively, and there are various partial results for affine-type groups (see \cite{aff1,aff2,Lee1,Lee2}, for example). Similarly, we refer the reader to \cite{Burness2020base,BGS0,BGS,BOW} for results towards a classification of the base-two almost simple groups.

Let $G \leqs {\rm Sym}(\O)$ be a base-two finite transitive permutation group with point stabiliser $H$. In \cite{SaxlGraph}, Burness and Giudici define the \emph{Saxl graph} $\Sigma(G)$ of $G$ as follows: the vertex set is $\O$ and two vertices are joined by an edge if and only if they form a base for $G$. Clearly, $\Sigma(G)$ is vertex-transitive and it is easy to see that $\Sigma(G)$ is connected if $G$ is primitive (as discussed in \cite{SaxlGraph}, if $G$ is imprimitive then $\Sigma(G)$ can have arbitrarily many connected components). In addition, $\Sigma(G)$ is a complete graph if and only if $G$ is a Frobenius group. Various problems concerning the valency and connectedness properties of Saxl graphs are investigated in \cite{SaxlGraph} and we refer the reader to \cite{ChenHuang2020val} for further results on the valency of $\Sigma(G)$ when $G$ is primitive. 

A number of open problems concerning Saxl graphs are presented in \cite{SaxlGraph} and the aim of this paper is to address some of these questions for the permutation groups in the collection denoted $\mathcal{G}$, which is defined as follows:
\[
\mathcal{G} = \{\mbox{\emph{finite almost simple primitive base-two groups with soluble point stabilisers}}\}.
\]
This is a natural collection to consider in this setting because one of the main results in \cite{Burness2020base} provides a complete classification of the groups in $\mathcal{G}$. We will also need to highlight the following subcollection, where $G_0$ denotes the socle of $G$ and $H$ is a point stabiliser:
\[
\mathcal{L} = \{ G \in \mathcal{G} \,:\, \mbox{$G_0 = {\rm L}_{2}(q)$, $H$ is of type ${\rm GL}_1(q) \wr S_2$ or ${\rm GL}_{1}(q^2)$} \}.
\]
In terms of Aschbacher's subgroup structure theorem \cite{asch}, the subgroups $H$ in the definition of $\mathcal{L}$ comprise the collections denoted $\mathcal{C}_2$ and $\mathcal{C}_3$. 
 
Perhaps the most intriguing open problem in this area is \cite[Conjecture 4.5]{SaxlGraph}, which asserts that if $G$ is primitive, then any two vertices in $\Sigma(G)$ have a common neighbour. In particular, this implies that the diameter of $\Sigma(G)$ is at most $2$ for every base-two finite primitive group $G$. Evidence for the conjecture is presented in \cite[Sections 4-6]{SaxlGraph}, where it is verified in several special cases. Our first main result establishes \cite[Conjecture 4.5]{SaxlGraph} for the groups in $\mathcal{G}$. This extends recent work of Chen and Du \cite{ChenDu2020Saxl}, who have shown that $\Sigma(G)$ has diameter $2$ for every base-two almost simple primitive group with socle ${\rm L}_{2}(q)$.

\begin{theorem}\label{t:main1}
Let $G \leqs {\rm Sym}(\O)$ be a permutation group in $\mathcal{G}$. Then any two vertices in $\Sigma(G)$ have a common neighbour. In particular, $\Sigma(G)$ has diameter $2$.
\end{theorem}

\begin{remk}
It is straightforward to extend our methods in order to establish \cite[Conjecture 4.5]{SaxlGraph} for all base-two almost simple primitive groups with socle ${\rm L}_{2}(q)$; we refer the reader to Theorem \ref{t:further}.
\end{remk}

Next we turn to the clique number of $\Sigma(G)$, which is denoted $\omega(G)$. Recall that this is the maximal size of a complete subgraph. By Theorem \ref{t:main1} we immediately deduce that $\omega(G) \geqs 3$ for every group $G$ in $\mathcal{G}$ and we can establish a stronger lower bound.

\begin{theorem}\label{t:main11}
Let $G \leqs {\rm Sym}(\O)$ be a permutation group in $\mathcal{G}$ with socle $G_0$ and assume that either $G \in \mathcal{G} \setminus \mathcal{L}$ or $G \leqs {\rm PGL}_{2}(q)$. Then $\omega(G) \geqs 4$, with equality if and only if $G = A_5$ and $\O$ is the set of $2$-element subsets of $\{1,\ldots, 5\}$.
\end{theorem}

\begin{remk}
Let us record a couple of comments on the statement of Theorem \ref{t:main11}.
\begin{itemize}\addtolength{\itemsep}{0.2\baselineskip}
\item[{\rm (a)}] First note that if $G = A_5$ and $\O$ is the set of $2$-element subsets of $\{1,\ldots, 5\}$, then $\Sigma(G)$ is the complement of the Petersen graph, which also coincides with the Johnson graph $J(5,2)$.
\item[{\rm (b)}] We expect that the bound $\omega(G) \geqs 5$ also holds for the groups $G \in \mathcal{L}$ with $G \not\leqs {\rm PGL}_{2}(q)$, but we have not been able to verify this in all cases. Here the main difficulty involves constructing an explicit clique of size $5$ when $G$ contains field automorphisms, working with a suitable geometric description of $\O$ (for example, if $H$ is of type ${\rm GL}_1(q) \wr S_2$, then we may identify $\O$ with the set of pairs of distinct $1$-dimensional subspaces of the natural module for $G_0$). With the aid of {\sc Magma} \cite{Magma}, we have verified the bound $\omega(G) \geqs 5$ for all groups in $\mathcal{L}$ with $5<q<1000$, but the general case remains open. We refer the reader to Remarks \ref{r:c2_clique} and \ref{r:c3_clique} for further details on the  difficulties that arise.
\end{itemize}
\end{remk}

For our next result, recall that the \emph{independence number} of $\Sigma(G)$, denoted $\a(G)$, is the maximal size of a co-clique (in other words, it is the clique number of the complement of $\Sigma(G)$). By applying work of Magaard and Waldecker \cite{MagaardIndependence2, MagaardIndependence3}, we will prove the following result.

\begin{theorem}\label{t:main12}
Let $G \leqs {\rm Sym}(\O)$ be a permutation group in $\mathcal{G}$. Then either $\a(G) \geqs 4$, or $G = A_5$, $\O$ is the set of $2$-element subsets of $\{1,\ldots, 5\}$ and $\a(G)=2$.
\end{theorem}

As explained in \cite[Remark 2.2]{SaxlGraph}, the Saxl graph $\Sigma(G)$ can be viewed as the \emph{generalised orbital graph} corresponding to the regular suborbits of $G$. In particular, $\Sigma(G)$ is an orbital graph of $G$ if and only if $G$ has a unique regular suborbit (recall that the base-two condition implies that $G$ has at least one regular suborbit), which in turn is equivalent to the property that $G$ acts arc-transitively on $\Sigma(G)$. The following result completely determines when a group in $\mathcal{G}$ has a unique regular suborbit.

\begin{theorem}\label{t:main2}
Let $G \leqs {\rm Sym}(\O)$ be a permutation group in $\mathcal{G}$ with point stabiliser $H$. Then $G$ has a unique regular suborbit if and only if 
\begin{itemize}\addtolength{\itemsep}{0.2\baselineskip}
\item[{\rm (i)}] $G = {\rm PGL}_{2}(q)$, $H = D_{2(q-1)}$ and $q \geqs 4$, $q \ne 5$; or
\item[{\rm (ii)}] $(G,H)$ is one of the cases listed in Table $\ref{tab:reg}$.
\end{itemize}
\end{theorem}

{\small
\begin{table}
\[
\begin{array}{lll} \hline
G & \mbox{Type of $H$} & \mbox{Comments} \\ \hline
{\rm P\Omega}_8^+(3).2^2 & {\rm O}_4^+(3)\wr S_2 & \mbox{Both groups of this shape} \\
\Omega_8^+(2).3 & {\rm O}_2^-(2) \times {\rm GU}_3(2) & \\
{\rm SO}_7(3) & {\rm O}_4^+(3) \perp {\rm O}_3(3) & \\
{\rm PSp}_6(3) & {\rm Sp}_2(3)\wr S_3 & \\ 
{\rm PGL}_4(3) & {\rm O}_4^+(3) & \\
{\rm U}_4(3).[4] & {\rm GU}_1(3)\wr S_4 & G \ne G_0.\la \delta^2,\phi\ra \\
{\rm U}_4(3) & {\rm GU}_2(3)\wr S_2 & \\
{\rm L}_3(4).D_{12} & {\rm GL}_1(4^3) & \\
{\rm L}_{3}(4).2 & \GU_3(2) & G \ne {\rm P\Sigma L}_{3}(4) \\
{\rm U}_3(5).S_3 & {\rm GU}_1(5)\wr S_3 & \\
{\rm U}_3(4) & {\rm GU}_1(4)\wr S_3 & \\
{\rm L}_2(11).2 & 2^{1+2}_{-}.{\rm O}_{2}^{-}(2) & \\
{\rm L}_{2}(9).2 & {\rm GL}_{1}(9^2) & G = {\rm M}_{10} \\ 
{\rm L}_{2}(5) & {\rm GL}_{1}(5^2) & \\
G_2(3).2 & {\rm SL}_2(3)^2 & \\
S_7 & {\rm AGL}_1(7) & \\
{\rm J}_2.2 & 5^2{:}(4\times S_3) & \\
{\rm M}_{11} & 2.S_4 & \\ \hline
\end{array}
\]
\caption{Groups in $\mathcal{G}$ with a unique regular suborbit}
\label{tab:reg}
\end{table}}

\begin{remk}\label{r:reg}
Some comments on the statement of Theorem \ref{t:main2} are in order:
\begin{itemize}\addtolength{\itemsep}{0.2\baselineskip}
\item[{\rm (a)}] In the second column of Table \ref{tab:reg} we record the \emph{type of $H$}. If $G$ is a group of Lie type, then this gives an approximate description of the structure of $H$, which is consistent with usage in \cite{KleidmanLiebeckClassicalGroups} when $G$ is a classical group. For example, the notation in the first row indicates that $H$ is the stabiliser in $G$ of an orthogonal decomposition $V = V_1 \perp V_2$ of the natural module for $G_0$, where the $V_i$ are nondegenerate $4$-dimensional subspaces of plus-type. In the final three rows of the table, we present the precise structure of $H$ in the second column.
\item[{\rm (b)}] In the first row of the table, there are two non-isomorphic groups $G$ of the form $G_0.2^2$, up to conjugacy in ${\rm Aut}(G_0)$. In both cases, $G$ has a unique regular suborbit.
\item[{\rm (c)}] Suppose $G = G_0.[4]$, where $G_0 = {\rm U}_{4}(3)$ and $H$ is of type $\GU_1(3) \wr S_4$. Note that there are three groups of this form up to conjugacy in ${\rm Aut}(G_0)$, namely $G_0.4 = {\rm PGU}_{4}(3)$ and two groups of shape $G_0.2^2$. As recorded in \cite[Table 8.10]{Low-Dimensional}, $H$ is not maximal when $G = G_0.\la \delta^2,\phi\ra = G_0.2^2$, but in the other two cases $H$ is maximal and $G$ has a unique regular suborbit (here we are adopting the notation for automorphisms from \cite{Low-Dimensional}).
\item[{\rm (d)}] Suppose $G = G_0.2$, where $G_0 = {\rm L}_{3}(4)$ and $H$ is of type $\GU_3(2)$. There are three groups of this form; each contains a maximal subgroup of the given type, but $b(G) = 2$ if and only if $G = G_0.2_1$ or $G_0.2_3$, and in both cases $G$ has a unique regular suborbit. Equivalently, $G = G_0.2$ has a unique regular suborbit if and only if $G \ne {\rm P\Sigma L}_{3}(4)$.
\end{itemize}
\end{remk}

Probabilistic and computational methods play a key role in the proofs of our main theorems and we will introduce the relevant results we will need in Section \ref{s:prel}. One of our main tools is Proposition \ref{p:bound}, which implies that if $G \leqs {\rm Sym}(\O)$ is a finite transitive base-two group and $Q(G)<1/4$, then $\omega(G) \geqs 5$ and any four vertices in $\Sigma(G)$ have a common neighbour. Here $Q(G)$ is the probability that a random pair of points in $\O$ do not form a base for $G$ (so the condition $b(G)=2$ implies that $Q(G)<1$). In view of Theorems \ref{t:main1} and \ref{t:main11}, this explains why one of our main aims is to establish the bound  $Q(G)<1/4$ whenever possible, using upper bounds on fixed point ratios to do this. This approach turns out to be effective for the groups in $\mathcal{G}\setminus \mathcal{L}$. Indeed, Theorem \ref{t:random} shows that if $G$ is such a group then $Q(G) \geqs 1/4$ if and only if $G$ is one of the finitely many groups recorded in Tables \ref{tab:random} and \ref{tab:random2} (this is the main content of Section \ref{s:random}). The latter groups are small enough to be amenable to computational methods and we refer the reader to Section \ref{ss:comp} for more details. 

However, the groups in $\mathcal{L}$ behave rather differently and this explains why they need to be handled separately (see Section \ref{ss:psl2}). 
For example, if $G = {\rm PGL}_{2}(q)$ and $H = D_{2(q-1)}$ then $G$ has a unique regular suborbit and we deduce that
\[
Q(G) = 1 - \frac{4(q-1)}{q(q+1)}.
\]
In particular, $Q(G)$ tends to $1$ as $q$ tends to infinity, whence Proposition \ref{p:bound} is not going to be useful in this situation. 

The proof of Theorem \ref{t:main1} will be completed in Section \ref{s:main1} and proofs of Theorems \ref{t:main11} and \ref{t:main12} are presented in Section \ref{s:main11}. Finally, the proof of Theorem \ref{t:main2} is given in Section \ref{s:main2}.

\vs

\noindent \textbf{Notation.} Our notation is fairly standard. Given a finite group $G$ and a positive integer $n$, we write $C_n$, or just $n$, for a cyclic group of order $n$ and $G^n$ for the direct product of $n$ copies of $G$. An unspecified extension of $G$ by a group $H$ will be denoted by $G.H$; if the extension splits then we write $G{:}H$. We use $[n]$ for an unspecified soluble group of order $n$. If $X$ is a subset of $G$, then $i_n(X)$ is the number of elements of order $n$ in $X$. We adopt the standard notation for simple groups of Lie type from \cite{KleidmanLiebeckClassicalGroups}. In particular, ${\rm L}_{n}^{\e}(q)$ denotes ${\rm PSL}_{n}(q)$ (when $\e=+$) and ${\rm PSU}_{n}(q)$ (when $\e=-$). We write ${\rm P\O}_{n}^{\e}(q)$ for the simple orthogonal groups, which differs from the notation used in the \textsc{Atlas} \cite{ATLAS}. All logarithms are base two.

\vs

\noindent \textbf{Acknowledgements.} Both authors thank Eamonn O'Brien for his assistance with several computations in this paper. The second author thanks the China Scholarship Council for supporting his doctoral studies at the University of Bristol and he thanks the Southern University of Science and Technology (SUSTech) for their generous hospitality during a visit in 2021.

\section{Preliminaries}\label{s:prel}

In this section, we present the main probabilistic and computational methods that will be used in the proofs of our main theorems.

\subsection{Probabilistic methods}\label{ss:prob}

Let $G \leqs {\rm Sym}(\O)$ be a transitive permutation group on a finite set $\O$ with point stabiliser $H$. Here we recall a powerful probabilistic approach for bounding the base size of $G$, which was originally introduced by Liebeck and Shalev \cite{LSh2} in their innovative proof of a conjecture of Cameron and Kantor on bases for almost simple primitive groups. 

Fix a positive integer $c$ and let $Q(G,c)$ be the probability that a randomly chosen $c$-tuple of points in $\O$ does not form a base for $G$, that is
\begin{equation}\label{e:qgc}
Q(G,c) = \frac{|\{(\a_1, \ldots, \a_c) \in \O^c \,:\, \bigcap_i G_{\a_i} \ne 1\}|}{|\O|^c}
\end{equation}
and note that $b(G) \leqs c$ if and only if $Q(G,c)<1$. Clearly, a subset $\{\a_1, \ldots, \a_c \} \subseteq \O$ is not a base for $G$ if and only if there exists an element $x \in G$ of prime order fixing each $\a_i$. Now the probability that $x$ fixes a randomly chosen element of $\O$ is given by the \emph{fixed point ratio}
\[
{\rm fpr}(x) = \frac{|C_{\O}(x)|}{|\O|} = \frac{|x^G \cap H|}{|x^G|},
\]
where $C_{\O}(x)$ is the set of fixed points of $x$ on $\O$, whence    
\[
Q(G,c) \leqs \sum_{x \in \mathcal{P}} {\rm fpr}(x)^c = \sum_{i=1}^{k} |x_i^{G}| \cdot {\rm fpr}(x_i)^c =:  \what{Q}(G,c),
\]
where $\mathcal{P} = \bigcup_ix_i^G$ is the set of elements of prime order in $G$.  

The following elementary result (\cite[Lemma 2.1]{Bur7}) is a useful tool for estimating $\what{Q}(G,c)$.

\begin{lem}\label{l:calc}
Suppose $x_{1}, \ldots, x_{m}$ represent distinct $G$-classes such that $\sum_{i}{|x_{i}^{G}\cap H|}\leqs A$ and $|x_{i}^{G}|\geqs B$ for all $i$. Then 
\[
\sum_{i=1}^{m} |x_i^{G}| \cdot {\rm fpr}(x_i)^c \leqs B(A/B)^c
\]
for every positive integer $c$.
\end{lem}

In this paper, we are interested in the case $c=2$ and we define 
\begin{equation}\label{e:Qhat}
Q(G) := Q(G,2), \; \; \what{Q}(G) := \what{Q}(G,2)
\end{equation}
so $b(G) = 2$ if and only if $Q(G)<1$. As noted in \cite[Section 3.3]{SaxlGraph}, we have 
\begin{equation}\label{e:QG}
Q(G) = 1 - \frac{r|H|}{n}
\end{equation}
where $n = |G:H| = |\O|$ and $r \geqs 1$ is the number of regular suborbits of $G$.

Suppose $b(G) = 2$ and set 
\[
t(G) = \max\{m \in \mathbb{N}\,:\, Q(G)<1/m\}.
\]
The following key result is \cite[Lemma 3.6]{SaxlGraph}.

\begin{prop}\label{p:bound}
Let $G \leqs {\rm Sym}(\O)$ be a transitive group on a finite set $\O$ with $t(G) \geqs 2$. Then the following hold:
\begin{itemize}\addtolength{\itemsep}{0.2\baselineskip}
\item[{\rm (i)}] Any $t(G)$ vertices in the Saxl graph $\Sigma(G)$ have a common neighbour.
\item[{\rm (ii)}] The clique number of $\Sigma(G)$ is at least $t(G)+1$.
\end{itemize} 
\end{prop}

In particular, the conclusion to Theorem \ref{t:main1} holds if $Q(G)<1/2$. Similarly, the bound $\omega(G) \geqs 5$ in Theorem \ref{t:main11} on the clique number of $\Sigma(G)$ holds if $Q(G)<1/4$. Let us also observe that if $Q(G) < 1/4$ then $G$ has a unique regular suborbit only if 
\[
|H|^2 \leqs |G| < \frac{4}{3}|H|^2.
\]

By definition, each group  $G$ in the collection $\mathcal{G}$ satisfies the bound $Q(G)<1$ and in view of the above observations, one of our main aims is to show that $Q(G)<1/4$ in most cases (see Theorem \ref{t:random}). This will reduce the proofs of Theorems \ref{t:main1} and \ref{t:main11} to a small number of specific cases that require closer attention.

It turns out that there are two infinite families in $\mathcal{G}$ that will require special attention. This is the collection denoted $\mathcal{L}$ in the introduction, which comprises the groups with socle $G_0 = {\rm L}_{2}(q)$ and point stabiliser $H$ of type ${\rm GL}_{1}(q) \wr S_2$ or ${\rm GL}_{1}(q^2)$. More precisely, $\mathcal{L}$ also contains any groups in $\mathcal{G}$ that are permutation isomorphic to one of these almost simple primitive groups with socle ${\rm L}_2(q)$. In particular, let us observe that if $G \in \mathcal{G} \setminus \mathcal{L}$ then 
\begin{equation}\label{e:excl}
G_0 \not\in \{A_5, A_6, {\rm L}_{3}(2), {\rm Sp}_{4}(2)', {}^2G_2(3)'\} \cup \{ {\rm L}_2(q) \,: \, q \leqs 9\}. 
\end{equation}

\begin{rem}\label{r:asymp}
The rationale for isolating the groups in $\mathcal{L}$ becomes clear when we consider the probability $Q(G)$ defined above. For example, by inspecting the relevant proofs in \cite{Burness2020base}, it is easy to check that if $G \in \mathcal{G} \setminus \mathcal{L}$ then $Q(G) \to 0$ as $|G| \to \infty$ (in particular, there are only finitely many groups of this form with $Q(G) \geqs 1/4$). However, the groups in $\mathcal{L}$ behave differently and they need to be handled separately. For example, if $G = {\rm L}_2(q)$ and $q$ is odd, then by inspecting the proofs of \cite[Lemmas 4.7 and 4.8]{Burness2020base} we see that
\[
Q(G) = \left\{ \begin{array}{ll}
1 - \frac{(q-1)(q+a)}{2q(q+1)} & \mbox{if $H = D_{q-1}$} \\
& \\
1 - \frac{(q+1)(q-b)}{2q(q-1)} & \mbox{if $H = D_{q+1}$}
\end{array}\right.
\]
where $(a,b) = (7,1)$ if $q \equiv 1 \imod{4}$, otherwise $(a,b)=(5,3)$. Therefore, in both cases we get $Q(G) \to 1/2$ as $q$ tends to infinity. Similarly, as noted in Section \ref{s:intro}, if $G = {\rm PGL}_2(q)$ then $Q(G) \to 1$ when $q$ is odd and $H = D_{2(q-1)}$. As a consequence, we cannot appeal to Proposition \ref{p:bound} in these cases and we will need to adopt a constructive approach to establish our main results (see Sections \ref{ss:psl2_c2} and \ref{ss:psl2_c3}).
\end{rem}

\subsection{Computational methods}\label{ss:comp}

We will apply computational methods to establish our main results when $G$ is a sporadic group, or a small degree symmetric or alternating group, or a low rank group of Lie type defined over a suitably small field. We mainly use {\sc Magma} V2.25-7 \cite{Magma} to do the computations, noting that the \textsf{GAP} Character Table Library \cite{GAPCTL} is an important tool for the analysis of sporadic groups.  

Let $G \leqs {\rm Sym}(\O)$ be a base-two almost simple primitive group with socle $G_0$ and soluble point stabiliser $H$. Our initial aim is to construct $G$ as a permutation group of an appropriate degree (this will not necessarily be the permutation representation of $G$ on $\O$). To do this, we typically use the {\sc Magma} function \texttt{AutomorphismGroupSimpleGroup} to obtain $A = {\rm Aut}(G_0)$ as a permutation group and then we identify $G$ by inspecting the subgroups of $A$ containing $G_0$. We then construct $H$ as a subgroup of $G$ in this permutation representation via the command \texttt{MaximalSubgroups(G:IsSolvable:=true)}, which returns a set of representatives of the $G$-classes of soluble maximal subgroups of $G$. In a handful of cases (due to the size of $G$), this approach is ineffective and a different method is needed in order to construct $H$. For example, we may identify $H = N_G(K)$ for some specific $p$-subgroup $K$ of $G$ (for instance, see \cite[Example 2.4]{Burness2020base}). 

The next key step is to estimate $Q(G)$, with the aim of determining the groups with $Q(G) \geqs 1/4$ (this is the content of Theorem \ref{t:random}). First we focus on $\what{Q}(G)$, recalling that $Q(G) \leqs \what{Q}(G)$ (see \eqref{e:Qhat}). It is straightforward to implement an algorithm in {\sc Magma} to compute $\what{Q}(G)$ precisely, using the functions \texttt{ConjugacyClasses} and \texttt{IsConjugate} to find a set of representatives of the conjugacy classes in $H$ and to test conjugacy in $G$, respectively. This allows us to compute $|x^G \cap H|$ for each $x \in H$ of prime order, which is the main step in calculating the contribution to $\what{Q}(G)$ from the elements in the $G$-class of $x$. Note that this approach can be implemented  without determining a set of representatives of the conjugacy classes in $G$, which can be an expensive operation in terms of time and memory. Let us also observe that $\what{Q}(G)$ can be computed precisely if we have access to the character tables of $G$ and $H$, in addition to the fusion map from $H$-classes to $G$-classes. For example, this approach works well when $G$ is a sporadic group, using the character table data stored in the \textsf{GAP} Character Table Library \cite{GAPCTL} (see the proof of Proposition \ref{p:spor}).

In some cases, it turns out that we can work effectively with a crude bound 
\begin{equation}\label{e:Qtilde}
\what{Q}(G) \leqs \widetilde{Q}(G)
\end{equation}
where the contribution to $\widetilde{Q}(G)$ from all the elements in $G$ of order $r$ (for a fixed prime $r$) with $|x^G| = m$ is given by 
\[
\frac{1}{m}\left(\sum_{i=1}^\ell|y_i^H|\right)^2
\]
and $y_1, \ldots, y_{\ell}$ represent the distinct $H$-classes of elements of order $r$ with $|y_i^G|=m$. Notice that no \texttt{IsConjugate} commands are needed to compute $\widetilde{Q}(G)$, which can be a significant saving. 

If $\what{Q}(G) \geqs 1/4$ then we need to either compute a better upper bound on $Q(G)$, or we need to determine $Q(G)$ precisely. In view of \eqref{e:QG}, it suffices to bound (or compute) the number $r$ of regular suborbits of $G$. To do this, we work with double cosets, noting that if $R$ is a complete set of $(H,H)$ double coset representatives in $G$, then  
\begin{equation}\label{e:reg}
r = |\{ x \in R \,:\, |HxH| = |H|^2\}|.
\end{equation}
If $|G:H|$ is not prohibitively large (for example, if $|G:H|< 10^7$), then we can use the {\sc Magma} function \texttt{DoubleCosetRepresentatives} to determine $R$ and then compute $r$. If $|G:H|$ is large, then we may be able to use the \texttt{DoubleCosetCanonical} function to identify sufficiently many distinct double cosets of size $|H|^2$ so that the corresponding lower bound on $r$ forces $Q(G)<1/4$. Similarly, we may be able to find a set $S$ of distinct $(H,H)$ double coset representatives so that 
\[
\sum_{x \in S} |HxH| > |G| - |H|^2
\]
and thus $r = |\{x \in S \,:\, |HxH| = |H|^2\}|$. It is straightforward to implement all of these methods in {\sc Magma}.

\section{Random bases}\label{s:random}

In this section we will prove the following theorem, which will be a key ingredient in the proofs of Theorems \ref{t:main1} and \ref{t:main11}. We adopt the notation $\mathcal{G}$ and $\mathcal{L}$ introduced above. In particular, we recall that the socle $G_0$ of a group in $\mathcal{G} \setminus \mathcal{L}$ satisfies the restrictions in \eqref{e:excl}.

\begin{thm}\label{t:random}
Let $G \leqs {\rm Sym}(\O)$ be a permutation group in $\mathcal{G} \setminus \mathcal{L}$ with socle $G_0$ and point stabiliser $H$. Then either $Q(G)<1/4$, or $(G,H)$ is one of the cases in Tables $\ref{tab:random}$ and $\ref{tab:random2}$.
\end{thm}

{\small
\begin{table}
\[
\begin{array}{llll} \hline
G & H & r & Q(G) \\ \hline
A_9 & {\rm ASL}_{2}(3) & 2 & 17/35 \\ 
S_7 & {\rm AGL}_{1}(7) & 1 & 13/20 \\
{\rm M}_{11} & 2.S_4 & 1 & 39/55 \\
{\rm M}_{12} & A_4 \times S_3 & 13 & 16/55 \\
{\rm M}_{12}.2 & S_4 \times S_3 & 4 & 31/55 \\
{\rm J}_{1} & 19{:}6 & 9 & 257/770 \\
& D_6 \times D_{10} & 34 & 443/1463 \\
{\rm J}_{2} & 5^2{:}D_{12} & 2 & 59/84 \\
{\rm J}_{2}.2 & 5^2{:}(4 \times S_3) & 1 & 59/84 \\
{\rm J}_{3}.2 & 3^{2+1+2}{:}8.2 & 3 & 886/1615 \\
& 2^{2+4}{:}(S_3 \times S_3) & 10 & 457/969 \\
{\rm HS}.2 & 5^{1+2}.[2^5] & 3 & 106/231 \\
{\rm McL}.2 & 2^{2+4}{:}(S_3 \times S_3) & 228 & 9419/28875 \\ 
{\rm He}.2 & 2^{4+4}.(S_3 \times S_3).2 & 5 & 23011/29155 \\
{\rm Suz} & 3^{2+4}{:}2.(A_4 \times 2^2).2 & 16 & 7529/25025 \\
{\rm Suz}.2 & 3^{2+4}{:}2.(S_4 \times D_8) & 4 & 16277/25025 \\
{\rm HN} & 5^{1+4}{:}2^{1+4}.5.4 & 47 & 332152/1066527 \\
{\rm HN}.2 & 5^{1+4}{:}2^{1+4}.5.4.2 & 22 & 34457/96957 \\
{}^2B_2(8) & 13{:}4 & 7 & 7/20 \\
{}^2B_2(8).3 & 13{:}12 & 2 & 31/70 \\
{}^2F_4(2)' & 5^2{:}4A_4 & 6 & 27/52 \\
{}^2F_4(2) & 5^2{:}4S_4 & 3 & 27/52 \\
G_2(3) & ({\rm SL}_{2}(3) \circ {\rm SL}_{2}(3)).2 & 4 & 563/819 \\ 
G_2(3).2 & ({\rm SL}_{2}(3) \circ {\rm SL}_{2}(3)).2.2 & 1 & 691/819 \\
\hline
\end{array}
\]
\caption{The groups in $\mathcal{G} \setminus \mathcal{L}$ with $Q(G) \geqs 1/4$, part I}
\label{tab:random}
\end{table}}

{\small
\begin{table}
\[
\begin{array}{llll} \hline
G & \mbox{Type of $H$} & r & Q(G) \\ \hline
{\rm L}_{4}(3) & {\rm O}_{4}^{+}(3) & 6 & 131/195 \\
{\rm L}_{4}(3).2_3 & {\rm O}_{4}^{+}(3) & 3 & 131/195 \\
{\rm L}_{4}(3).2_2 & {\rm O}_{4}^{+}(3) & 2 & 457/585 \\
{\rm L}_{4}(3).2_1 = {\rm PGL}_{4}(3) & {\rm O}_{4}^{+}(3) & 1 & 521/585 \\
{\rm L}_{3}(9).2^2 & \GL_1(9) \wr S_3 & 48 & 4093/12285 \\
{\rm L}_{3}(5) & {\rm GL}_{1}(5) \wr S_3 & 30 & 199/775 \\
{\rm L}_{3}(5).2 & {\rm GL}_{1}(5) \wr S_3 & 13 & 1379/3875 \\
 & {\rm GL}_{1}(5^3) & 13 & 791/2000 \\
{\rm L}_{3}(4).2 \ne {\rm P\Sigma L}_{3}(4) &  \GU_3(2) & 1 & 17/35 \\
{\rm L}_{3}(4).6 &  \GL_1(4^3) & 5 & 11/32 \\
{\rm L}_{3}(4).S_3 = G_0.\la \delta,\gamma\ra &  \GL_1(4^3) & 3 & 97/160 \\
{\rm L}_{3}(4).D_{12} &  \GL_1(4^3) & 1 & 59/80 \\
{\rm L}_{3}(3) & \GL_1(3^3) & 2 & 11/24 \\
& {\rm O}_{3}(3) & 5 & 19/39 \\
{\rm L}_{3}(3).2 & {\rm O}_{3}(3) & 2 & 23/39 \\
{\rm L}_{2}(17) & 2_{-}^{1+2}.{\rm O}_{2}^{-}(2) & 3 & 5/17 \\
{\rm L}_{2}(13).2 & 2_{-}^{1+2}.{\rm O}_{2}^{-}(2) & 2 & 43/91 \\
{\rm L}_{2}(11).2 & 2_{-}^{1+2}.{\rm O}_{2}^{-}(2) & 1 & 31/55 \\
{\rm U}_4(5).2^2 & \GU_1(5) \wr S_4 & 409 & 361747/1421875 \\
{\rm U}_4(4) & \GU_1(4) \wr S_4 & 80 & 259/884 \\
{\rm U}_4(4).2 & \GU_1(4) \wr S_4 & 30 & 1661/3536 \\
{\rm U}_4(4).4 & \GU_1(4) \wr S_4 & 15 & 1661/3536 \\
{\rm U}_4(3) & \GU_2(3) \wr S_2 & 1 & 187/315 \\
{\rm U}_4(3).2 = {\rm U}_4(3).\la \delta^2\phi \ra & \GU_1(3) \wr S_4 & 4 & 1811/2835 \\
{\rm U}_4(3).2^2 \ne {\rm U}_{4}(3).\la \delta^2,\phi\ra & \GU_1(3) \wr S_4 & 1 & 2323/2835  \\
{\rm U}_4(3).4 & \GU_1(3) \wr S_4 & 1 & 2323/2835 \\
{\rm U}_3(9).2 & \GU_1(9) \wr S_3 & 40 & 1913/5913 \\
{\rm U}_{3}(9).4 & \GU_1(9) \wr S_3 & 20 & 1913/5913 \\
{\rm U}_3(8).2 & \GU_1(8) \wr S_3 & 78 & 1097/4256 \\
{\rm U}_3(8).S_3 & \GU_1(8) \wr S_3 & 19 & 205/448 \\
{\rm U}_3(8).6 & \GU_1(8) \wr S_3 & 25 & 2437/8512 \\
{\rm U}_{3}(8).(3 \times S_3) & \GU_1(8) \wr S_3 & 6 & 2069/4256 \\
{\rm U}_{3}(7) & \GU_1(7) \wr S_3 & 27 & 4381/14749 \\
{\rm U}_{3}(7).2 & \GU_1(7) \wr S_3 & 10 & 7069/14749 \\
{\rm U}_{3}(5).3 & \GU_1(5) \wr S_3 & 3 & 551/875 \\
& 3^{1+2}.{\rm Sp}_{2}(3) & 5 & 67/175 \\
{\rm U}_{3}(5).S_3 & \GU_1(5) \wr S_3 & 1 & 659/875 \\
& 3^{1+2}.{\rm Sp}_{2}(3) & 2 & 443/875 \\
{\rm U}_3(4) & \GU_1(4) \wr S_3 & 1 & 133/208 \\
{\rm PSp}_6(3) & {\rm Sp}_{2}(3) \wr S_3 & 1 & 853/1365 \\
{\rm Sp}_{4}(4).4 & {\rm O}_{2}^{-}(4) \wr S_2 & 2 & 103/153 \\
{\rm P\O}_8^{+}(3) & {\rm O}_{4}^{+}(3) \wr S_2 & 12 & 45041/61425 \\
{\rm P\O}_8^{+}(3).2 = {\rm PSO}_{8}^{+}(3) & {\rm O}_{4}^{+}(3) \wr S_2 & 4 & 151507/184275 \\
{\rm P\O}_8^{+}(3).2  = {\rm P\Omega}_{8}^{+}(3).\la \gamma \ra & {\rm O}_{4}^{+}(3) \wr S_2 & 3 & 53233/61425 \\
{\rm P\O}_8^{+}(3).3 & {\rm O}_{4}^{+}(3) \wr S_2 & 3 & 16379/20475 \\ 
{\rm P\O}_8^{+}(3).2^2 & {\rm O}_{4}^{+}(3) \wr S_2 & 1 & 167891/184275 \\
{\rm P\O}_8^{+}(3).4 & {\rm O}_{4}^{+}(3) \wr S_2 & 2 & 151507/184275  \\
{\rm P\O}_8^{+}(3).S_4 & {\rm O}_{2}^{-}(3) \wr S_4 & 823 &  17810761/44778825 \\
\O_8^{+}(2).3 & {\rm O}_{2}^{-}(2) \times \GU_3(2) & 1 & 2071/2800 \\
\O_7(3) & {\rm O}_{4}^{+}(3) \perp {\rm O}_{3}(3) &  5 & 1945/2457 \\
{\rm SO}_7(3) & {\rm O}_{4}^{+}(3) \perp {\rm O}_{3}(3) &  1 & 11261/12285 \\
\hline
\end{array}
\]
\caption{The groups in $\mathcal{G} \setminus \mathcal{L}$ with $Q(G) \geqs 1/4$, part II}
\label{tab:random2}
\end{table}}

\begin{rem}\label{r:q14}
We make several comments concerning Theorem \ref{t:random}.
\begin{itemize}\addtolength{\itemsep}{0.2\baselineskip}
\item[{\rm (a)}] In Table \ref{tab:random} we record the relevant cases where $G_0$ is non-classical (as previously noted, the condition $G \in \mathcal{G}\setminus \mathcal{L}$ implies that $G_0 \ne A_5,A_6$ or ${}^2G_2(3)'$).
\item[{\rm (b)}] The cases with $G_0$ classical are listed in Table \ref{tab:random2}. As before, the \emph{type of $H$} provides an approximate description of the group theoretic structure of $H$ (the precise structure can be read off from \cite{KleidmanLiebeckClassicalGroups}). 
\item[{\rm (c)}] In both tables, the number of regular suborbits of $G$ is listed in the third column.
\item[{\rm (d)}] We use the standard \textsc{Atlas} \cite{ATLAS} notation for describing the almost simple groups of the form ${\rm L}_{4}(3).2$. In particular, ${\rm L}_{4}(3).2_2$ and ${\rm L}_{4}(3).2_3$ contain involutory graph automorphisms $x$ with $C_{G_0}(x) = {\rm PSp}_{4}(3).2$ and ${\rm PSO}_{4}^{-}(3).2$, respectively.
\item[{\rm (e)}] Suppose $G = G_0.S_3$, where $G_0 = {\rm L}_3(4)$ and $H$ is of type ${\rm GL}_{1}(3^4)$. There are two groups of this form, up to conjugacy in ${\rm Aut}(G_0)$, and we find that $r=6$ and $Q(G) = 17/80$ if $G = G_0.\la \delta, \phi \ra$, whereas $r=3$ and $Q(G) = 97/160$ if $G = G_0.\la \delta, \gamma\ra$. Here we are using the notation for automorphisms in \cite{Low-Dimensional}, where $\delta$, $\phi$ and $\gamma$ denote diagonal, field and graph automorphisms, respectively. We adopt similar notation to describe the relevant groups with $G_0 = {\rm U}_{4}(3)$ or ${\rm P\O}_{8}^{+}(3)$ (in the latter case, $\gamma$ is an involutory graph automorphism).
\end{itemize}
\end{rem}

\subsection{Alternating and sporadic groups}\label{s:altspor}

\begin{prop}\label{p:alt}
The conclusion to Theorem \ref{t:random} holds if $G_0$ is an alternating group.
\end{prop}

\begin{proof}
Let $G_0 = A_m$ be the socle of $G$. If $m \leqs 12$ then the result is easily checked using {\sc Magma} \cite{Magma} (see Section \ref{ss:comp}), so let us assume $m \geqs 13$. By inspecting \cite[Table 14]{SolubleStabiliser} and \cite[Table 4]{Burness2020base} we deduce that $m$ is a prime and $H = {\rm AGL}_{1}(m) \cap G$, in which case  
\[
|H| \leqs m(m-1) = a, \;\; |x^G| \geqs  \frac{m!}{((m-1)/2)!2^{(m-1)/2}} = b
\]
for all $x \in H$ of prime order (minimal if $x$ is an involution, noting that $x$ has at most one fixed point on $\{1, \ldots, m\}$). In view of Lemma \ref{l:calc}, this gives $\what{Q}(G) < a^2/b < 1/4$ and the result follows. 
\end{proof}

\begin{prop}\label{p:spor}
The conclusion to Theorem \ref{t:random} holds if $G_0$ is a sporadic group.
\end{prop}

\begin{proof}
First recall that the maximal subgroups of $G$ have been classified up to conjugacy, with the exception of the Monster group $\mathbb{M}$, where the problem of determining all the almost simple maximal subgroups is still open. In particular, the possibilities for $(G,H)$ are known and \cite{Wilson} is a convenient reference. In addition, the groups with $b(G) \geqs 3$ are listed in \cite[Table 4]{Burness2020base}.
 
First assume $G$ is not the Baby Monster $\mathbb{B}$ nor the Monster $\mathbb{M}$. In the remaining cases we first use the \textsf{GAP} Character Table Library \cite{GAPCTL} to identify the relevant groups with $\what{Q}(G) \geqs 1/4$. Indeed, in each case the character table of $G$ is available in \cite{GAPCTL} and we can use the \texttt{Maxes} function to access the character table of the maximal subgroup $H$. Moreover, \cite{GAPCTL} also stores the fusion map from $H$-classes to $G$-classes and this allows us to compute precise fixed point ratios and subsequently determine the exact value of $\what{Q}(G)$.

This reduces the problem to a small number of cases that require further attention. To handle these groups, we adopt the method described in Section \ref{ss:comp} to compute $Q(G)$ precisely. First we use the function \texttt{AutomorphismGroupSimpleGroup} to construct $G$ as a permutation group and we obtain $H$ via the \texttt{MaximalSubgroups} function (for $G_0 = {\rm HN}$ and $H \cap G_0 = 5^{1+4}.2^{1+4}.5.4$ we construct $H$ using the generators given in the Web Atlas \cite{WebAt}). Next we use \texttt{DoubleCosetRepresentatives} to construct a complete set $R$ of $(H,H)$ double coset representatives and this allows us to calculate $Q(G)$ via \eqref{e:QG} and \eqref{e:reg} (we thank Eamonn O'Brien for his assistance with this computation in the special case where $G_0 = {\rm HN}$ and $H \cap G_0 = 5^{1+4}.2^{1+4}.5.4$). In this way, we can read off the groups with $Q(G) \geqs 1/4$ and they are recorded in Table \ref{tab:random}.

Finally, suppose $G = \mathbb{B}$ or $\mathbb{M}$. If $G = \mathbb{B}$ then $H = [3^{11}].(S_4 \times 2S_4)$ or $47{:}23$; in both cases we can use \cite{GAPCTL} and the \texttt{Maxes} function as above to show that $\what{Q}(G)<1/4$. Similarly, if $G = \mathbb{M}$ then $H = 13^{1+2}{:}(3 \times 4S_4)$ or $41{:}40$ and once again we can use \cite{GAPCTL} to verify the bound $\what{Q}(G)<1/4$ (here we use \texttt{NamesOfFusionSources} in place of \texttt{Maxes} since the latter is not available for $\mathbb{M}$).
\end{proof}

\subsection{Exceptional groups}\label{s:excep}

Next let us turn to the groups in $\mathcal{G}\setminus \mathcal{L}$ where $G_0$ is an exceptional group of Lie type over $\mathbb{F}_q$ with $q=p^f$ for a prime $p$. As noted in the proof of \cite[Proposition 7.1]{Burness2020base}, the condition $b(G) = 2$ implies that $H$ is a maximal rank subgroup (that is, $H$ contains a maximal torus of $G$). More precisely, either $H = N_G(T)$ for some maximal torus $T$ of $G_0$ (see \cite[Table 5.2]{MaxExceptional}), or $(G,H)$ is one of the cases recorded in Table \ref{tab:ex}. Recall that $G_0 \ne {}^2G_2(3)' \cong {\rm L}_{2}(8)$ since $G \in \mathcal{G}\setminus \mathcal{L}$.

{\small
\begin{table}
\[
\begin{array}{lll} \hline
& G_0 & \mbox{Type of $H$} \\ \hline
{\rm (a)} & G_2(3) & {\rm SL}_2(3)^2 \\ 
{\rm (b)} & {}^3D_4(2) & 3\times {\rm SU}_3(2) \\
{\rm (c)} & {}^2F_4(2)' & {\rm SU}_3(2) \\
{\rm (d)} & F_4(2) & {\rm SU}_3(2)^2 \\
{\rm (e)} & {}^2E_6(2) & {\rm SU}_3(2)^3 \\ 
{\rm (f)} & E_8(2) & {\rm SU}_3(2)^4 \\ \hline
\end{array}
\]
\caption{The groups in $\mathcal{G} \setminus \mathcal{L}$, $G_0$ exceptional, $H \ne N_G(T)$}
\label{tab:ex}
\end{table}}

\begin{lem}\label{l:ex1}
The conclusion to Theorem \ref{t:random} holds if $G_0$ is an exceptional group of Lie type and $H$ is the normaliser of a maximal torus.
\end{lem}

\begin{proof}
The possibilities for $H$ are recorded in \cite[Table 5.2]{MaxExceptional} and \cite[Proposition 4.2]{ExtremelyPrimitive} states that $b(G)=2$ whenever $H$ is the normaliser of a maximal torus (soluble or otherwise). We proceed by carefully inspecting the proof of \cite[Proposition 4.2]{ExtremelyPrimitive} in the relevant cases with $H$ soluble. 

If $G_0 = E_8(q)$ then one checks that the bound on $\what{Q}(G)$ in the proof of \cite[Lemma 4.3]{ExtremelyPrimitive} is sufficient and we note that $H$ is insoluble when $G_0 = E_7(q)$. 

Next assume $G_0 = E_6^{\e}(q)$. Here the proof of \cite[Lemma 4.11]{ExtremelyPrimitive} yields $\what{Q}(G)<q^{-1}$ if $q \geqs 5$, so we may assume $q \leqs 4$ and 
\[
N_L(H_0) = (q^2 +\e q +1)^3.3^{1+2}.{\rm SL}_{2}(3),
\]
where $H_0 = H \cap G_0$, $L = {\rm Inndiag}(G_0)$ and $(q,\e) \ne (2,-)$. One checks that the upper bound on $\what{Q}(G)$ presented in the proof of \cite[Lemma 4.11]{ExtremelyPrimitive} is sufficient unless $\e=+$ and $q \leqs 3$. If $q=3$ then $H$ does not contain any long root elements (see \cite[Corollary 2.13]{ExtremelyPrimitive}, for example) and by bounding the contribution to $\what{Q}(G)$ from the remaining elements of prime order, as in the proof of \cite[Lemma 4.11]{ExtremelyPrimitive}, we deduce that $\what{Q}(G)<1/4$. 

Now suppose $(q,\e) = (2,+)$. As explained in the proof of \cite[Lemma 4.11]{ExtremelyPrimitive}, we can use {\sc Magma} to construct $H$ as a subgroup of $E_7(8)$ (see \cite[Example 1.11]{BTh_comp} for the details). Let $i_m(H)$ denote the number of elements in $H$ of order $m$. If $x \in H$ has odd prime order then $|x^G|>2^{31}=b_1$ and we calculate that $i_3(H) = 11438$ and $i_7(H) = 342$, so Lemma \ref{l:calc} implies that the contribution to $\what{Q}(G)$ from elements of odd prime order is less than $a_1^2/b_1$, where $a_1 = 11780$. Now assume $x \in H$ is an involution. We find that $H_0$ contains $a_2 = 441$ involutions, so the contribution from these elements is less than $a_2^2/b_2$, where $b_2 = 2^{21}$. Similarly, there are $a_3=406$ involutions $x \in H_0.2 \setminus H_0$, each of which acts on $G_0$ as a graph automorphism. Therefore $|x^G|>\frac{1}{3}2^{25}=b_3$ and we conclude that
\begin{equation}\label{e:q14}
\what{Q}(G) < \sum_{i=1}^3a_i^2/b_i < \frac{1}{4}.
\end{equation}

Next assume $G_0 = F_4(q)$, so $q$ is even and $G$ contains graph automorphisms (see \cite[Table 5.2]{MaxExceptional}). By applying the bounds on $\what{Q}(G)$ in the proof of \cite[Lemma 4.15]{ExtremelyPrimitive} we immediately reduce to the case $q=2$. Here $G = F_4(2).2$ and $H = 7^2{:}(3 \times 2.S_4)$ is the normaliser of a Sylow $7$-subgroup of $G$. The upper bound on $\what{Q}(G)$ in the proof of \cite[Lemma 4.15]{ExtremelyPrimitive} is larger than $1/2$, but we can use {\sc Magma} to construct $G$ and $H$ as permutation groups of degree $139776$ (more precisely, we use \texttt{AutomorphismGroupSimpleGroup} to construct $G$ and we find $H$ by taking the normaliser of a Sylow $7$-subgroup). Then by considering the fusion of $H$-classes in $G$ we calculate that 
\[
\what{Q}(G) = \frac{541861}{29328998400}
\]
and the result follows. 

Now suppose $G_0 = G_2(q)$, so $p=3$, $q \geqs 9$ and $G$ contains graph automorphisms (see \cite[Table 5.2]{MaxExceptional}). By arguing as in the proof of \cite[Lemma 4.21]{ExtremelyPrimitive} we reduce to the cases $q \in \{9,27\}$. Suppose $q=27$ and note that $|H_0| \leqs 12(q+1)^2 = a_1$. Let $x \in H$ be an element of prime order. If $x \in H_0$ then $|x^G| \geqs q^3(q^3-1)(q+1)=b_1$ (as noted in the proof of \cite[Lemma 4.21]{ExtremelyPrimitive}), whereas if $x$ is a field automorphism then $|x^G|>q^{28/3}=b_2$ and $H$ contains at most $a_2=24(q+1)^2$ such elements. Similarly, if $x$ is an involutory graph automorphism then $|x^G| =  q^3(q^3-1)(q+1)=b_3$ and there are at most $a_3=12(q+1)^2$ such elements in $H$.  It is straightforward to check that \eqref{e:q14} holds. Finally, suppose $q=9$. First we use {\sc Magma} to construct $G = G_2(9).4 = {\rm Aut}(G_0)$ as a permutation group of degree $132860$ and we note that $H = N_G(K)$, where $K$ is either a Sylow $\ell$-subgroup of $G_0$ (with $\ell \in \{5,13,73\}$) or $K = C_8 \times C_8$. In each case, it is straightforward to construct $H$ and verify the bound $\what{Q}(G)<1/4$. 

To complete the proof of the lemma, we may assume $G_0$ is one of the twisted groups ${}^3D_4(q)$, ${}^2F_4(q)'$, ${}^2G_2(q)$ ($q \geqs 27$) or ${}^2B_2(q)$. First assume $G_0 = {}^3D_4(q)$, in which case there are three possibilities for $H$ and one checks that the bound on $\what{Q}(G)$ in the proof of \cite[Lemma 4.24]{ExtremelyPrimitive} is sufficient if $q \geqs 9$. Suppose $q=8$ and let $x \in H$ be an element of prime order, which implies that $x \in H_0.3$. Then $|x^G|>8^{14}=b_1$ and $3|H_0| \leqs 383688 = a_1$, whence $\what{Q}(G)<a_1^2/b_1<1/4$. The same conclusion holds when $q=7$ since $|H| \leqs 233928$ and $|x^G|>7^{14}$ for all $x \in H$ of prime order. The remaining groups with $q \leqs 5$ can be handled using {\sc Magma}. In each case, we can use \texttt{AutomorphismGroupSimpleGroup} to construct $G$ and we obtain $H$ as the normaliser in $G$ of an appropriate Sylow $\ell$-subgroup of $G_0$. For example, if $q=5$ then the three possibilities for $H$ correspond to the primes $\ell \in \{7,31,601\}$. In every case, it is straightforward to verify the bound $\what{Q}(G)<1/4$.

Next assume $G_0= {}^2F_4(q)'$. The case $q=2$ can be checked using {\sc Magma} and we note that $Q(G)>1/4$ when $H \cap G_0 = 5^2{:}4A_4$ (as recorded in Table \ref{tab:random}). For $q \geqs 8$, the upper bound on $\what{Q}(G)$ in the proof of \cite[Lemma 4.26]{ExtremelyPrimitive} is sufficient (note that in the upper bound on $|H|$ given in the proof of this lemma, the $2\log q$ factor can be replaced by $|{\rm Out}(G_0)| = \log q$). The case $G_0 = {}^2G_2(q)$ is very similar. Indeed, if $q \geqs 3^5$ then the upper bound on $\what{Q}(G)$ in the proof of \cite[Lemma 4.37]{CameronConjecture} is good enough, while the case $q = 27$ can be handled using {\sc Magma}, noting that $H=N_G(K)$ with $K$ a Sylow $\ell$-subgroup of $G_0$ for $\ell \in \{7,19,37\}$. Finally, let us assume $G_0 = {}^2B_2(q)$. If $q \geqs 2^9$ then the bounds in the proof of \cite[Lemma 4.39]{CameronConjecture} are good enough. If $q=2^7$ and $x \in H$ is a field automorphism of order $7$ then $|x^G|>q^4$ and by arguing as in the proof of \cite[Lemma 4.39]{CameronConjecture} we deduce that $\what{Q}(G)<1/4$. The remaining cases with $q \in \{8,32\}$ can be checked using {\sc Magma} and we find that $Q(G)<1/4$ unless $q=8$ and $H \cap G_0 = 13{:}4$. The latter case is recorded in Table \ref{tab:random}.
\end{proof}

\begin{prop}\label{p:ex}
The conclusion to Theorem \ref{t:random} holds if $G_0$ is an exceptional group of Lie type.
\end{prop}

\begin{proof}
In view of the previous lemma, we may assume $G$ is one of the groups listed in Table \ref{tab:ex}.  In cases (a), (b) and (c) we can use {\sc Magma} to prove the result (we get $Q(G)<1/4$ in cases (b) and (c), while $Q(G) \geqs 1/4$ in (a)). In (d), the upper bound on $\what{Q}(G)$ in the proof of \cite[Lemma 4.16]{ExtremelyPrimitive} is insufficient. But as explained in \cite[Example 1.4]{BTh_comp}, we can use {\sc Magma} to construct $G$ and $H$ and then it is straightforward to verify the bound $\what{Q}(G)<1/4$.

Finally, let us consider cases (e) and (f). In (e) we observe that $2$ and $3$ are the only prime divisors of $|H|$ and we deduce that $\what{Q}(G)<1/4$ by applying the relevant bounds presented in Case 1 in the proof of \cite[Lemma 4.12]{ExtremelyPrimitive}. Similarly, in case (f) we note that the only prime divisors of $|H|$ are $2$ and $3$. If $x$ is a long root element, then 
\[
|x^G\cap H| = 4|y^L| = 36 = a_1, \;\; |x^G|>2^{58}=b_1,
\]
where $y$ is a long root element in $L = {\rm SU}_3(2)$. If not, then $|x^G| > 2^{92} = b_2$ and we note that $|H| = 104485552128 = a_2$. By applying Lemma \ref{l:calc} we deduce that  
\[
\what{Q}(G)< a_1^2/b_1 + a_2^2/b_2 < \frac{1}{4}
\]
and the result follows.
\end{proof}

In order to complete the proof of Theorem \ref{t:random}, we may assume $G_0$ is a classical group. It will be convenient to partition the proof into various subsections according to the socle $G_0$. The cases that we need to consider are recorded in the following result, which is an immediate consequence of \cite{SolubleStabiliser} and \cite[Theorem 2]{Burness2020base}. 

\begin{thm}\label{t:baseclass}
Let $G \leqs {\rm Sym}(\O)$ be a permutation group in $\mathcal{G} \setminus \mathcal{L}$ with socle $G_0$ classical and point stabiliser $H$. Then $(G,H)$ is one of the cases in Table $\ref{tab:cases}$.
\end{thm}

\begin{table}
\[
\begin{array}{lll} \hline
G_0 & \mbox{Type of $H$} & \mbox{Conditions} \\ \hline
{\rm L}_n(q) & {\rm GL}_1(q^n) & \mbox{$n \geqs 3$ prime, $G \ne {\rm L}_{3}(3).2$} \\
& {\rm GL}_{2}(q) \wr S_{n/2} & \mbox{$n \in \{6,8\}$, $q=3$} \\
& {\rm GL}_1(q)\wr S_n & n \in \{3,4\}, \, q \geqs 5 \\
& {\rm O}_4^+(q) & \mbox{$(n,q)=(4,3)$, $G\ne {\rm Aut}(G_0)$} \\
& {\rm O}_3(q) & (n,q) = (3,3) \\
& 3^{1+2}.{\rm Sp}_2(3) & \mbox{$n=3$, $p=q\equiv 1 \imod{3}$} \\
& {\rm GL}_{2}(3) & \mbox{$n=2$, $q=3^f$, $f \geqs 3$ prime} \\
& 2_{-}^{1+2}.{\rm O}_{2}^{-}(2) & \mbox{$n=2$, $q=p \geqs 11$} \\ 
\UU_n(q) & \GU_1(q^n) & \mbox{$n \geqs 3$ prime} \\
& \GU_{1}(q) \wr S_{n} & \mbox{$n \in \{3,4\}$, $q \geqs 3$, $(n,q) \ne (3,3)$} \\
& \GU_{2}(q) \wr S_{n/2} & \mbox{$n \in \{4,6,8\}$, $q=3$} \\
& \GU_{3}(q) \wr S_{n/3} & \mbox{$n \in \{9,12\}$, $q=2$} \\
& 3^{1+2}.\Sp_2(3) & \mbox{$n=3$, $q=p\equiv 2\imod{3}$} \\
& \GU_3(2) & \mbox{$n=3$, $q=2^f$, $f\geqs 3$ prime} \\
\PSp_n(q) & \Sp_2(q)\wr S_{n/2} & \mbox{$n \in \{6,8\}$, $q=3$} \\ 
& {\rm O}_2^\epsilon(q)\wr S_2 & \mbox{$(n,p)=(4,2)$, $q \geqs 4$} \\
& {\rm O}_{2}^{-}(q^2) & \mbox{$(n,p)=(4,2)$, $q \geqs 4$} \\ 
\POmega_{n}^+(q) & {\rm O}_4^+(q)\wr S_{n/4} & \mbox{$n \in \{12,16\}$, $q=3$} \\ 
& {\rm O}_4^+(q) \wr S_2 & \mbox{$(n,q)=(8,3)$, $|G:G_0|<6$} \\
& {\rm O}_2^{\e}(q)\wr S_4 & \mbox{$n=8$, $q \geqs 3$} \\
& {\rm O}_2^{-}(q) \times \GU_3(q) & \mbox{$(n,q)=(8,2)$, $G=G_0.3$} \\
& {\rm O}_2^-(q^2) \times {\rm O}_2^-(q^2) & n = 8 \\
\O_n(q) & {\rm O}_4^+(q) \perp {\rm O}_3(q) & (n,q) = (7,3) \\ \hline
\end{array}
\]
\caption{The groups in $\mathcal{G}\setminus \mathcal{L}$ with $G_0$ classical}\label{tab:cases}
\end{table}

Note that in the final column of Table \ref{tab:cases} we list necessary conditions for the existence of a group  in $\mathcal{G}\setminus \mathcal{L}$ with the given socle and point stabiliser. In general, these conditions are not sufficient and  we refer the reader to \cite{Low-Dimensional,KleidmanLiebeckClassicalGroups} and \cite{Burness2020base}  for a precise description of the conditions that are needed to ensure that $G$ has a maximal subgroup of the given type and a base of size $2$. We also refer the reader to \cite[Chapter 3]{BG_book} for detailed information on the conjugacy classes of elements of prime order in $G$.

\subsection{Linear groups}\label{ss:lin}

In this section we assume $G_0 = {\rm L}_{n}(q)$. Recall that the condition $G \not\in \mathcal{L}$ implies that $q \geqs 11$ if $n=2$, and $q \geqs 3$ if $n=3$. 

\begin{prop}\label{p:lin}
The conclusion to Theorem \ref{t:random} holds if $G_0 = {\rm L}_{n}(q)$.
\end{prop}

\begin{proof}
First assume $n$ is a prime and $H$ is of type ${\rm GL}_{1}(q^n)$. By applying the upper bound on $\what{Q}(G)$ in the proof of \cite[Lemma 6.4]{Burness2020base} we immediately reduce to the cases where  $(n,q) = (7,2)$, or $n=5$ and $q \leqs 5$, or $n=3$ and $q \leqs 19$. With the aid of {\sc Magma}, it is straightforward to compute $Q(G)$ precisely in each of these cases and the result quickly follows (note that the condition $b(G)=2$ implies that $G \ne {\rm L}_{3}(3).2$). In particular, we find that $Q(G) \geqs 1/4$ only if $n = 3$ and $q \leqs 5$ (the precise exceptions are recorded in Table \ref{tab:random2}).

Next assume $n \in \{3,4\}$, $q \geqs 5$ and $H$ is of type ${\rm GL}_1(q)\wr S_n$. First assume $n=3$. By  inspecting the proof of \cite[Lemma 6.5]{Burness2020base} we deduce that $\what{Q}(G)<1/4$ if $q \geqs 43$. If $29 \leqs q \leqs 41$ then $G$ does not contain field automorphisms of order $2$ or $3$, nor graph-field automorphisms of order $2$, so we may set $a_8 = a_9=0$ in the bound on $\what{Q}(G)$ presented in the proof of  \cite[Lemma 6.5]{Burness2020base}. One checks that this modified bound yields $\what{Q}(G)<1/4$. For $7 \leqs q \leqs 27$ we can use {\sc Magma} to show that $Q(G)<1/4$ in the usual manner, with the single exception of the case $G = \Aut(G_0)$ with $q=9$, where $Q(G) = 4093/12285$. Finally, for $q=5$ we calculate that $Q(G) = 199/775$ if $G = G_0$, otherwise $Q(G) = 1379/3875$; both cases are recorded in Table \ref{tab:random2}. Similarly, if $n=4$ then the result follows by combining explicit {\sc Magma} computations for $q \in \{5,7,8\}$ with the upper bound on $\what{Q}(G)$ presented in the proof of \cite[Lemma 6.6]{Burness2020base} for $q \geqs 9$ (in every case we get $Q(G)<1/4$). The case where $n=3$ and $H$ is of type $3^{1+2}.{\rm Sp}_2(3)$ is entirely similar, working with the bound on $\what{Q}(G)$ in the proof of \cite[Lemma 6.11]{Burness2020base}. 

Now let us turn to the relevant groups with $G_0 = {\rm L}_{2}(q)$. If $q=3^f$ with $f \geqs 3$ a prime, then the bound on $\what{Q}(G)$ in the proof of \cite[Lemma 4.9]{Burness2020base} is sufficient if $f \geqs 7$, while the cases $f \in \{3,5\}$ are easily checked using {\sc Magma}. Similarly, if $q=p \geqs 11$ and $H$ is of type $2^{1+2}_{-}.{\rm O}_{2}^{-}(2)$ then the bound in the proof of \cite[Lemma 4.10]{Burness2020base} is good enough if $q \geqs 71$ and we can use {\sc Magma} to handle the cases with $q<71$.

There are four remaining cases to consider. If $G_0 = {\rm L}_3(3)$ with $H$ of type ${\rm O}_3(3)$ then we compute $Q(G) = 19/39$ if $G = G_0$, whereas $Q(G) = 23/39$ for $G = G_0.2$. Next suppose $G_0 = {\rm L}_{4}(3)$ and $H$ is of type ${\rm O}_{4}^{+}(3)$. Here the condition $b(G) = 2$ implies that $G \ne {\rm Aut}(G_0)$ and using {\sc Magma} one checks that $Q(G)>1/4$ in each case (the precise value of $Q(G)$ is recorded in Table \ref{tab:random2}). The case $G_0 = {\rm L}_{6}(3)$ with $H$ of type ${\rm GL}_{2}(3) \wr S_3$ can be handled using {\sc Magma}, working with the function \texttt{MaximalSubgroups} to construct $H$. Finally, suppose $G_0 = {\rm L}_{8}(3)$ and $H$ is of type ${\rm GL}_{2}(3) \wr S_4$. Here the \texttt{MaximalSubgroups} function is ineffective, but we can construct $H$ by observing that $H = N_G(K)$ with $|K|=2^{11}$, as noted in the proof of \cite[Proposition 6.3]{Burness2020base} (also see \cite[Example 2.4]{Burness2020base}). It is straightforward to check that $\what{Q}(G)<1/4$.
\end{proof}

\subsection{Unitary groups}\label{ss:unit}

\begin{prop}\label{p:unit}
The conclusion to Theorem \ref{t:random} holds if $G_0 = {\rm U}_{n}(q)$ with $n \geqs 3$.
\end{prop}

\begin{proof}
First assume $n$ is a prime and $H$ is of type $\GU_1(q^n)$. For $n \geqs 5$, one checks that the upper bound on $\what{Q}(G)$ in the proof of \cite[Lemma 6.4]{Burness2020base} is sufficient unless $n=5$ and $q \leqs 5$. Suppose $n=5$, so $q \geqs 3$ by the maximality of $H$. If $q=5$ then it is easy to improve the given bound in \cite{Burness2020base}  in order to show that $\what{Q}(G)<1/4$ (for example, we can use the fact that $|H| \leqs 10(5^5+1)/6$). For $q=4$ we observe that $H = N_G(P)$, where $P$ is a Sylow $41$-subgroup of $G$, so it is straightforward to construct $H$ in {\sc Magma} and verify the bound $\what{Q}(G)<1/4$ (note that it suffices to check this for $G = \Aut(G_0)$). The case $q=3$ can also be checked using {\sc Magma} (using \texttt{MaximalSubgroups} to construct $H$, or noting that $H$ is the normaliser of a Sylow $61$-subgroup). Similarly, if $n=3$ then $q \geqs 4$ and the bound in the proof of \cite[Lemma 6.4]{Burness2020base} is sufficient for $q \geqs 23$ (for $q=32$, we note that $|x^G| \geqs |G_0:{\rm U}_{3}(2)|$ if $x$ is a field automorphism of order $5$); the remaining cases with $q \leqs 19$ can be verified using {\sc Magma}.

Next suppose $G_0 = {\rm U}_{3}(q)$ and $H$ is of type $\GU_1(q) \wr S_3$ with $q \geqs 4$. For $q \geqs 43$ it is easy to check that the upper bound on $\what{Q}(G)$ in the proof of \cite[Lemma 6.5]{Burness2020base} is sufficient. The same estimates are also good enough when $29 \leqs q \leqs 41$, noting that in each case $G$ does not contain any field or graph-field automorphisms of order $2$ or $3$.
For $11 \leqs q \leqs 27$ we can use {\sc Magma} to verify the bound $\what{Q}(G)<1/4$. We find that there are examples with $Q(G) \geqs 1/4$ when $q \leqs 9$; they are easily identified using {\sc Magma} and they are recorded in Table \ref{tab:random2}. The case where $G_0 = {\rm U}_{4}(q)$ and $H$ is of type $\GU_1(q) \wr S_4$ is similar. Here $q \geqs 3$ and the bound on $\what{Q}(G)$ in the proof of \cite[Lemma 6.5]{Burness2020base} is good enough for $q \geqs 9$. If $q \in \{7,8\}$ then one can check that $Q(G)<1/4$ using {\sc Magma}. In the same way, we find that there are exceptions to this bound when $q \in \{3,4,5\}$ and each of these cases is listed in Table \ref{tab:random2}.

Next let us turn to the groups where $n \in \{4,6,8\}$, $q=3$ and $H$ is of type $\GU_2(q) \wr S_{n/2}$. If $n=4$ then $G = G_0$ is the only group with $b(G)=2$ (see \cite[Table 7]{Burness2020base}) and with the aid of {\sc Magma} we calculate that $Q(G) = 187/315$. Next assume $n=6$. Here $H = N_G(K)$ for some subgroup $K$ of $G_0$ of order $2^{10}$ and it is straightforward to check that $\what{Q}(G)<1/4$ (see \cite[Example 2.4]{Burness2020base} and the proof of \cite[Proposition 6.3]{Burness2020base}). Similarly, if $n=8$ then $H = N_G(K)$ with $|K|=2^{13}$ and once again one checks that $\what{Q}(G)<1/4$. (Note that in both cases, it suffices to check the bound for $G = \Aut(G_0)$.)

Now assume $n \in \{9,12\}$, $q=2$ and $H$ is of type $\GU_3(q) \wr S_{n/3}$. As noted in the proof of \cite[Proposition 6.3]{Burness2020base}, if $n=9$ then $H = N_G(K)$ with $|K|=3^8$ and we can use {\sc Magma} to verify the bound $\what{Q}(G)<1/4$. 

For $n=12$ we find that the bound presented in the proof of \cite[Proposition 6.3]{Burness2020base} does not give $\what{Q}(G)<1/4$ and a more accurate estimate is required. To do this, it suffices to improve the upper bound on the contribution to $\what{Q}(G)$ from elements of order $3$. 

As in the proof of \cite[Proposition 6.3]{Burness2020base}, we may view $H$ as the stabiliser in $G$ of an orthogonal decomposition 
\[
V = V_1 \perp V_2 \perp V_3 \perp V_4
\]
of the natural module, where each $V_i$ is a nondegenerate $3$-space. Suppose $x \in H$ has order $3$. If some conjugate of $x$ induces a nontrivial permutation of the $V_i$, then $|x^G|>2^{89}=b_1$ and we note that $|H|<2^{42}=a_1$. Following the argument in \cite{Burness2020base}, the contribution from the remaining elements of order $3$ in $H$ with $|x^G|>2^{69}=b_2$ is less than $a_2^2/b_2$, where $a_2 = 2^{31}$. As explained in the proof of \cite[Proposition 6.3]{Burness2020base}, the contribution from the elements with $|x^G| \leqs 3.2^{62}$ is less than $2\sum_{i=3}^7a_i^2/b_i$, where the integers $a_i$ and $b_i$ are defined as in the proof in \cite{Burness2020base}. Finally, if $3.2^{62} < |x^G| \leqs 2^{69}$ then one can check that $x$ is of the form $[I_8,\omega I_4]$, where $\omega \in \mathbb{F}_4$ is a primitive cube root of unity. Here we calculate 
\[
|x^G \cap H| \leqs 2\binom{4}{2}m + \binom{4}{2}m^2+2\binom{4}{2}m^3+m^4= 42480 = a_0
\]
where $m = \frac{1}{3}|\GU_3(2):\GU_2(2)| = 12$. Therefore, the contribution to $\what{Q}(G)$ from elements of order $3$ is less than
\[
a_1^2/b_1+a_2^2/b_2+2\left(a_0^2/b_0+\sum_{i=3}^7a_i^2/b_i\right)<\frac{1}{20}
\]
where $b_0 = 3.2^{62}$. Finally, the estimates in the proof of \cite[Proposition 6.3]{Burness2020base} imply that the contribution to $\what{Q}(G)$ from involutions is also less than $1/20$ and the result follows.

To complete the proof of the proposition, we may assume $n=3$ and either $q=p \equiv 2 \imod{3}$ and $H$ is of type $3^{1+2}.{\rm Sp}_{2}(3)$, or $q=2^f$ with $f \geqs 3$ a prime and $H$ is a subfield subgroup of type $\GU_3(2)$. Suppose $H$ is of type $3^{1+2}.{\rm Sp}_{2}(3)$. Here the proof of \cite[Lemma 6.11]{Burness2020base} gives the result for $q > 29$ and we can use {\sc Magma} to handle the cases with $q \leqs 29$, noting that there are exceptions to the bound $Q(G)<1/4$ when $q=5$ (as recorded in Table \ref{tab:random2}). Finally, let us assume $H$ is of type $\GU_3(2)$, so $q=2^f$ with $f \geqs 3$ odd. If $f \geqs 7$ then the bound on $\what{Q}(G)$ in the proof of \cite[Lemma 6.10]{Burness2020base} is sufficient, while the cases with $f \in \{3,5\}$ can be handled using {\sc Magma}.
\end{proof}

\subsection{Symplectic groups}\label{ss:symp}

Next assume $G_0 = {\rm PSp}_{n}(q)$ with $n \geqs 4$. Recall that $(n,q) \ne (4,2)$ since $G \not\in \mathcal{L}$. 

\begin{prop}\label{p:symp}
The conclusion to Theorem \ref{t:random} holds if $G_0 = {\rm PSp}_{n}(q)$ with $n \geqs 4$.
\end{prop}

\begin{proof}
First assume $n \in \{6,8\}$, $q=3$ and $H$ is of type $\Sp_2(q)\wr S_{n/2}$. If $n=6$ then the condition $b(G)=2$ implies that $G = G_0$ and using {\sc Magma} we calculate that $Q(G) = 853/1365$, so this case is listed in Table \ref{tab:random2}. For $n=8$ we use \texttt{AutomorphismGroupSimpleGroup} and \texttt{MaximalSubgroups} to construct $G$ and $H$, and we apply \texttt{DoubleCosetCanonical} to establish the existence of sufficiently many regular $H$-orbits in order to force $Q(G)<1/4$ (see \eqref{e:QG}). Indeed, for $G = G_0$ we get $r \geqs 3113$, while $r \geqs 1557$ for $G = G_0.2$. 

Finally let us assume $G_0 = \Sp_4(q)$ with $q \geqs 4$ even and $H$ of type ${\rm O}_2^\epsilon(q)\wr S_2$ or ${\rm O}_{2}^{-}(q^2)$. Here $H$ is maximal only if $G$ contains graph automorphisms and with the aid of {\sc Magma} one checks that if $q \leqs 2^5$ then either $\what{Q}(G)<1/4$ or $q=4$, $G = \Aut(G_0)$ and $H$ is of type ${\rm O}_2^{-}(q)\wr S_2$. In the latter case we have $Q(G) = 103/153$ as recorded in Table \ref{tab:random2}. For the remainder, we may  assume $q \geqs 2^6$. 

Suppose $H$ is of type ${\rm O}_2^\epsilon(q)\wr S_2$, so $H_0 = (C_{q-\e})^2{:}D_8$. By applying the upper bound in the proof of \cite[Lemma 6.9]{Burness2020base}, we deduce that $\what{Q}(G)<1/4$ if $q \ne 2^7$. So let us assume $q=2^7$ and write $\what{Q}(G) = \a_1+\a_2$, where $\a_1$ is the contribution from involutory graph automorphisms. The proof of \cite[Lemma 6.9]{Burness2020base} gives $\a_2<2^{-6}$, so it remains for us to estimate $\a_1$. If $\e=-$ then $H \leqs (C_{129})^2{:}(SD_{16} \times C_7)$ and it follows that every involution in $H$ is contained in $H \cap G_0 = (C_{129})^2{:}D_{8}$, whence $\a_1=0$ and the result follows. Now assume $\e=+$, so $H \leqs (C_{127})^2{:}(D_{16} \times C_7)$. Since there are exactly $4$ involutions in $D_{16} \setminus D_8$, we deduce that $\a_1 \leqs d^2/b$ with $d = 4.127^2$ and $b = |G_0: {}^2B_2(q)| = 34626060288$. One checks that the resulting bound on $\what{Q}(G)$ is good enough.

To complete the proof, let us assume $q \geqs 2^6$ and $H$ is of type ${\rm O}_{2}^{-}(q^2)$, so  
\[
H_0 = {\rm O}_{2}^{-}(q^2).2 = C_{q^2+1}{:}C_4
\]
and we will estimate the contribution to $\what{Q}(G)$ from the various elements of prime order (the details in this case were omitted in the proof of \cite[Lemma 6.9]{Burness2020base}). First let $x \in H$ be a unipotent involution. Then $x$ embeds in $G$ as an involution of type $c_2$ (in the notation of Aschbacher and Seitz \cite{AS}), whence 
\[
|x^G \cap H| = i_2(H_0) = q^2+1=a_1,\;\; |x^G| = (q^2-1)(q^4-1) = b_1.
\]
If $x$ is semisimple, then $|x^G| \geqs |{\rm Sp}_{4}(q):\GU_1(q^2)| = q^4(q^2-1)^2 = b_2$ and we note that there are at most $a_2=q^2+1$ such elements in $H$. Next suppose $x$ is a field automorphism of odd order. Then $|x^G|>q^{20/3}=b_3$ and $H$ contains fewer than $4(q^2+1)\log q = a_3$ such elements. Finally, suppose $x$ is an involutory field or graph automorphism (note that $G$ cannot contain elements of both types). If $\log q$ is even then every involution in $H$ is contained in $H_0$, so we may assume $\log q$ is odd and $x$ is a graph automorphism. Then $|x^G| = q^2(q+1)(q^2-1) = b_4$ and we note that $|x^G \cap H| \leqs |H_0| = 4(q^2+1) = a_4$. Therefore, by applying Lemma \ref{l:calc} we deduce that
\[
\what{Q}(G) < \sum_{i=1}^3a_i^2/b_i + \a a_4^2/b_4,
\]
where $\a=1$ if $\log q$ is odd, otherwise $\a =0$, and we conclude that $\what{Q}(G)<1/4$.
\end{proof}

\subsection{Orthogonal groups}\label{ss:orth}

In order to complete the proof of Theorem \ref{t:random}, we may assume $G_0 = {\rm P\O}_{n}^{\e}(q)$ with $n \geqs 7$. 

\begin{prop}\label{p:orth}
The conclusion to Theorem \ref{t:random} holds if $G_0 = {\rm P\O}_{n}^{\e}(q)$ with $n \geqs 7$.
\end{prop}

\begin{proof}
By inspecting Table \ref{tab:cases} we observe that either $n$ is even and $\e=+$, or $(n,q) = (7,3)$. First assume $n \in \{12,16\}$, $q=3$ and $H$ is of type ${\rm O}_4^+(q)\wr S_{n/4}$. For $n=16$, the upper bound in the proof of \cite[Proposition 6.3]{Burness2020base} gives $\what{Q}(G)<1/4$. On the other hand, if $n=12$ then we can construct $G$ and $H$ in {\sc Magma} (see \cite[Example 2.4]{Burness2020base}) and it is straightforward to check that $\what{Q}(G)<1/4$. The relevant cases with $G_0 = \O_7(3)$ or $\O_8^{+}(2)$ can also be handled using {\sc Magma} and the exceptions with $Q(G) \geqs 1/4$ are recorded in Table \ref{tab:random2}. 

To complete the proof, we may assume $G_0 = {\rm P\Omega}_{8}^{+}(q)$ with $q \geqs 3$. Suppose $q=3$ and $H$ is of type ${\rm O}_{4}^{+}(3) \wr S_2$, noting that $|G:G_0|<6$ since $b(G)=2$. Even though $|G:H|=14926275$ is large, we can still analyse this case in the usual way using {\sc Magma}, working with a set of $(H,H)$ double coset representatives to compute $r$ (and hence $Q(G)$) via \eqref{e:reg}. The results are presented in Table \ref{tab:random2}.

Next assume $H$ is of type ${\rm O}_{2}^{\e'}(q) \wr S_4$. If $q \in \{3,4\}$ then $\e'=-$ and using {\sc Magma} one can check that either $Q(G)<1/4$, or $q = 3$, $G = \Aut(G_0)$, $r=823$ and 
\[
Q(G) = \frac{17810761}{44778825}.
\]
For example, if $q=3$ and $G = G_0.A_4$ then using \texttt{DoubleCosetCanonical} we can verify the bound $r \geqs 3075$, which forces $Q(G)<1/4$. We thank Eamonn O'Brien for his assistance with the precise calculation of $r$ when $G = \Aut(G_0)$. For $q \geqs 5$, we seek to apply the upper bound on $\what{Q}(G)$ presented in the proof of \cite[Lemma 6.7]{Burness2020base}. If $q \geqs 9$ then
\[
\what{Q}(G) < 2q^{-1}+ q^{-2} + q^{-3}+q^{-7} <\frac{1}{4}
\]
and the result follows. One can check that the bounds in the proof of \cite[Lemma 6.7]{Burness2020base} are also sufficient when $q \in \{7,8\}$, so we may assume $q=5$. Here we have $H = N_G(K)$, where  $K<G_0$ has order $2^9$ if $\e'=+$, otherwise $|K|=3^4$. We now construct $H$ as in \cite[Example 2.4]{Burness2020base} and one checks that $\what{Q}(G)<1/4$.

Finally, let us assume $H$ is of type ${\rm O}_2^-(q^2) \times {\rm O}_2^-(q^2)$ with $q \geqs 3$. If $q \geqs 11$ then the upper bound on $\what{Q}(G)$ in the proof of \cite[Lemma 6.8]{Burness2020base} is sufficient. On the other hand, if $q \leqs 9$ then we can construct $H$ in {\sc Magma}, noting that $H = N_G(K)$ with $K$ a Sylow $\ell$-subgroup of $G_0$ and $\ell$ an odd prime divisor of $q^2+1$. In this way, it is straightforward to check that $\what{Q}(G)<1/4$ and the result follows.
\end{proof}

\vs

This completes the proof of Theorem \ref{t:random}.

\section{Two-dimensional linear groups}\label{ss:psl2}

In this section we turn to the groups in $\mathcal{L}$, so $G_0 = {\rm L}_{2}(q)$, $q \geqs 4$ and $H$ is of type ${\rm GL}_{1}(q) \wr S_2$ or ${\rm GL}_{1}(q^2)$. Note that these special cases coincide with the $\mathcal{C}_2$-actions and $\mathcal{C}_3$-actions of $G$, respectively, where a point stabiliser $H$ is a maximal subgroup of $G$ in the collection labelled $\C_i$ in Aschbacher's subgroup structure theorem \cite{asch}. The goal of this section is to establish our main theorems in these cases (we will handle Theorem \ref{t:main12} in Section \ref{ss:indep}). 

A key ingredient in our proof of Theorem \ref{t:main1} for the groups in $\mathcal{L}$ is the following recent result of Chen and Du \cite{ChenDu2020Saxl}.

\begin{thm}[Chen \& Du, \cite{ChenDu2020Saxl}]\label{t:cd}
Let $G \leqs {\rm Sym}(\O)$ be a finite almost simple primitive group with socle $G_0 = {\rm L}_{2}(q)$ and $b(G) = 2$. Then the Saxl graph $\Sigma(G)$ has diameter $2$.
\end{thm}

This establishes a special case of a conjecture in \cite{SaxlGraph}, which asserts that $\Sigma(G)$ has diameter at most $2$ for every finite primitive permutation group $G$ with $b(G)=2$. In fact, \cite[Conjecture 4.5]{SaxlGraph} states that the following even stronger property holds in this general setting:
\begin{equation}\label{e:star}
\mbox{\emph{Any two vertices in $\Sigma(G)$ have a common neighbour.}} \tag{$\star$}
\end{equation}
In view of Theorem \ref{t:cd}, in order to establish this for the groups we are considering in this section, it suffices to show that if $\{\a,\b\}$ is a base for $G$, then there exists $\gamma \in \O$ such that $\{\a,\gamma\}$ and $\{\b,\gamma\}$ are bases. 

Let us fix some notation. Let $V$ be the natural module for $G_0$ and write $q=p^f$, where $p$ is a prime. Fix a basis $\{e_1,e_2\}$ for $V$ and write $\mathbb{F}_{q}^{\times} = \la \mu \ra$. Let $\delta \in {\rm PGL}_{2}(q)$ be the image (modulo scalars) of the diagonal matrix ${\rm diag}(\mu, 1) \in {\rm GL}_{2}(q)$, which induces a diagonal automorphism on $G_0$. Similarly, let $\phi$ be a field automorphism of order $f$ such that $(ae_1+be_2)^{\phi} = a^pe_1+b^pe_2$ for all $a,b \in \mathbb{F}_q$ and note that 
\[
{\rm Aut}(G_0) = \la G_0, \delta, \phi \ra
\] 
and ${\rm P\Sigma L}_{2}(q) = \la G_0, \phi \ra$.
For $g \in {\rm Aut}(G_0)$, if we write $\ddot{g}$ for the coset $G_0g$, then
\[
{\rm Out}(G_0) = \{ \ddot{g} \,:\, g \in {\rm Aut}(G_0) \} = \la \ddot{\delta} \ra \times \la \ddot{\phi} \ra = C_{(2,q-1)} \times C_{f}.
\]
As before, if $H$ is a subgroup of $G$, then we set $H_0 = H \cap G_0$.

It is convenient to use computational methods to handle the cases where $q$ is small. To this end, we present the following result. Note that ${\rm L}_{2}(9).2 = {\rm M}_{10}$ in part (ii)(b). Also recall that we write $\omega(G)$ for the clique number of $\Sigma(G)$.

\begin{prop}\label{p:comp}
Let $G \leqs {\rm Sym}(\O)$ be a finite almost simple primitive group with socle $G_0 = {\rm L}_{2}(q)$ and point stabiliser $H$ of type ${\rm GL}_{1}(q) \wr S_2$ or ${\rm GL}_{1}(q^2)$. If $b(G) = 2$ and $q \leqs 27$, then the following hold:
\begin{itemize}\addtolength{\itemsep}{0.2\baselineskip}
\item[{\rm (i)}] Property \eqref{e:star} holds.
\item[{\rm (ii)}] $G$ has a unique regular suborbit if and only if one of the following holds:

\vspace{1mm}

\begin{itemize}\addtolength{\itemsep}{0.2\baselineskip}
\item[{\rm (a)}] $G={\rm PGL}_{2}(q)$, $q \geqs 4$, $q \ne 5$ and $H = D_{2(q-1)}$; or
\item[{\rm (b)}] $(G,H) = ({\rm L}_{2}(5),D_6)$ or $({\rm L}_{2}(9).2, 5{:}4)$. 
\end{itemize}

\item[{\rm (iii)}] $\omega(G) \geqs 4$, with equality if and only if $G = {\rm L}_{2}(4) \cong {\rm L}_{2}(5)$ and $H=D_6$. 
\end{itemize}
\end{prop}

\begin{proof}
To verify \eqref{e:star}, we proceed as in Section \ref{s:main1}, working with the {\sc Magma} functions \texttt{AutomorphismGroupSimpleGroup} and \texttt{MaximalSubgroups} to construct $G$ and $H$. We then use \texttt{DoubleCosetRepresentatives} to determine a set $R$ of $(H,H)$ double coset representatives and for each $x \in R$ we find an element $y \in G$ such that $H \cap H^y = H^x \cap H^y = 1$, which establishes \eqref{e:star}. In the same way, we can count the number of elements $x \in R$ with $|HxH| = |H|^2$, which coincides with the number of regular suborbits of $G$ (the existence of a unique regular suborbit in (ii)(a) was noted in \cite[Example 2.5]{SaxlGraph}). Finally, we can use the {\sc Magma} code presented in Section \ref{ss:clique} to verify the bound on $\omega(G)$ in part (iii).
\end{proof}

\subsection{$\C_2$-actions}\label{ss:psl2_c2}

Here $H$ is of type ${\rm GL}_{1}(q) \wr S_2$, so $H_0 = D_{2(q-1)/h}$ and $|\O| = \frac{1}{2}q(q+1)$, where $h=(2,q-1)$. We may identify $\O = G/H$ with the set of unordered pairs of distinct $1$-dimensional subspaces of the natural module $V$ for $G_0$. The maximality of $H$ implies that $q \geqs 4$ and $q \ne 5$ (see \cite[Table 8.1]{Low-Dimensional}, for example); in view of Proposition \ref{p:comp}, we may assume that $q>27$. By \cite[Lemma 4.7]{Burness2020base} we have $b(G,H) \leqs 3$, with equality if and only if ${\rm PGL}_{2}(q)<G$.  

As noted in \cite[Example 2.5]{SaxlGraph}, if $G=\PGL_2(q)$ then $\Sigma(G)$ is isomorphic to the \emph{Johnson graph} $J(q+1,2)$; the vertices of this graph correspond to the $2$-element subsets of a set of size $q+1$, with two vertices joined by an edge if they have nonempty intersection. This observation immediately implies that \eqref{e:star} holds, $G$ has a unique regular suborbit and $\Sigma(G)$ has clique number $q$. Therefore, for the remainder of this section we will assume that $q$ is odd and $G \cap {\rm PGL}_{2}(q) = G_0$. Then as noted in the proof of \cite[Lemma 4.7]{Burness2020base}, this implies that one of the following holds:
\begin{itemize}\addtolength{\itemsep}{0.2\baselineskip}
\item[{\rm (a)}] $G=\langle G_0,\phi^j\rangle$ for some $j$ in the range $0\leqs j<f$; or
\item[{\rm (b)}] $G=\langle G_0,\delta\phi^j\rangle$ with $0<j<f$ and $f/(f,j)$ even.
\end{itemize}

\vspace{2mm}

Set $\a,\b \in \O$, where $\a = \{\la e_1 \ra, \la e_2\ra\}$ and $\b = \{\la u \ra, \la v \ra \}$. Let us assume $q$ is odd and suppose $G = {\rm P\Sigma L}_{2}(q) = \la G_0, \phi \ra$. Notice that if $u = e_1$ and $v = be_1+e_2$, then $\a$ and $\b$ are fixed by the image in $G$ of an element 
\[
\begin{pmatrix}
	a&0\\
	0&a^{-1}
	\end{pmatrix}\phi \in \la {\rm SL}_{2}(q), \phi \ra
	\]
with $a^2 = b^{p-1}$. Similarly, the pointwise stabiliser of $\{\a,\b\}$ is nontrivial if $u = e_2$. Therefore, $\{\a,\b\}$ is a base for $G$ only if $\la u \ra = \la e_1+be_2\ra$ and $\la v \ra = \la e_1+ce_2\ra$ for distinct nonzero scalars $b,c \in \mathbb{F}_q$. 

In Lemma \ref{l:c2_cond_PSigmaL} below we present necessary and sufficient conditions on the scalars $b$ and $c$ to ensure that $\{\a,\b\}$ is a base for ${\rm P\Sigma L}_{2}(q)$. To do this, we need the following more general result. Note that the condition in part (iii) is equivalent to the non-containment of $bc^{-1}$ in a proper subfield of $\mathbb{F}_q$.

\begin{lem}\label{l:c2_cond}
Suppose $G\cap \PGL_2(q)=G_0$ with $q$ odd and set 
\[
\alpha=\{\la e_1\ra,\la e_2\ra\},\;\; \beta=\{ \la e_1+be_2 \ra, \la e_1+ce_2\ra\}
\]
with $b \ne c$. Then $\{\alpha,\beta\}$ is a base for $G$ if the following conditions are satisfied:
\begin{itemize}\addtolength{\itemsep}{0.2\baselineskip}
\item[{\rm (i)}] $bc \ne 0$; 
\item[{\rm (ii)}] $-bc^{-1}$ is a non-square in $\mathbb{F}_q$; and
\item[{\rm (iii)}] $b^{p^k-1}\ne c^{p^k-1}$ for all $0<k<f$.
\end{itemize}
\end{lem}

\begin{proof}
Suppose $b$ and $c$ satisfy the three given conditions and let us assume 
\[
x=AB^i\phi^j\in\la \GL_2(q),\phi\ra
\]
fixes $\alpha$ and $\beta$, where $A\in\SL_2(q)$, $B={\rm diag}(\mu,1)$, $0\leqs i<q-1$ and $0\leqs j<f$, with $j>0$ if $i>0$. It suffices to show that $x = \pm I_2$. Since $x$ fixes $\a$, the matrix of $A$ with respect to the basis $\{e_1,e_2\}$ is either diagonal or anti-diagonal. 

First assume $x$ fixes the two $1$-spaces comprising $\a$, so $A={\rm diag}(a,a^{-1})$ is diagonal. If $x$ also fixes the two spaces in $\beta$, then 
\begin{align*}
(e_1+be_2)^x&=a\mu^i e_1+a^{-1}b^{p^j}e_2=\eta_1(e_1+be_2)\\
(e_1+ce_2)^x&=a\mu^i e_1+a^{-1}c^{p^j}e_2=\eta_2(e_1+ce_2)
\end{align*}
for some $\eta_1,\eta_2\in\mathbb{F}_q^{\times}$. Therefore 
\begin{equation}\label{e:c21}
a^2\mu^i = b^{p^j-1} = c^{p^j-1}
\end{equation}
and thus (iii) implies that $j=0$, so $i=0$ and $a^2=1$, which gives $x = \pm I_2$ as required.
Similarly, if $x$ interchanges the spaces in $\beta$, then
\begin{equation}\label{e:c22}
a^2\mu^i = b^{p^j}c^{-1} = c^{p^j}b^{-1}.
\end{equation}
Here $b^{p^{2j}-1} = c^{p^{2j}-1}$, so (iii) implies that $2j=0$ or $f$. Suppose $2j=0$, so $i=0$ and $a^2 = bc^{-1} = cb^{-1}$ and thus $bc^{-1}=\pm 1$. But $b \ne c$, so $bc^{-1}=-1$, which is incompatible with (ii). Now assume $2j=f$, so $q \equiv 1 \imod{4}$ and $-1$ is a square in $\mathbb{F}_q$. In addition, \eqref{e:c22} gives $(bc^{-1})^{p^{f/2}+1} =1$, so $bc^{-1}\in\langle \mu^{p^{f/2}-1}\rangle$ and thus $bc^{-1}$ is a square. Therefore, $-bc^{-1}$ is a square, which once again is incompatible with (ii).  
	
	Now assume $A=\begin{pmatrix}
	0&a\\
	-a^{-1}&0
	\end{pmatrix}$ is anti-diagonal. If $x$ fixes both spaces in $\beta$ then
	\begin{align*}
	(e_1+be_2)^x&=ab^{p^j}e_1-a^{-1}\mu^ie_2=\eta_1(e_1+be_2)\\(e_1+ce_2)^x&=ac^{p^j}e_1-a^{-1}\mu^ie_2=\eta_2(e_1+ce_2)
	\end{align*}
	for some $\eta_1,\eta_2\in\mathbb{F}_q^{\times}$. This gives
	\begin{equation}\label{e:c23}
	-a^2\mu^{-i}=b^{-p^j-1}=c^{-p^j-1}.
	\end{equation}
	Here $b^{p^{2j}-1}=c^{p^{2j}-1}$ and thus $2j=0$ or $f$ by (iii). If $2j=0$ then $i=0$ and $-a^2=b^{-2}=c^{-2}$, which implies that $bc^{-1}=\pm 1$. As noted above, this is incompatible with (ii). Now assume $2j=f$, so $q\equiv 1\pmod 4$ and $-1$ is a square in $\mathbb{F}_q$ once again. Then (\ref{e:c23}) gives $(bc^{-1})^{p^{f/2}+1}=1$ and as above we deduce that $bc^{-1}$ is a square. Hence, $-bc^{-1}$ is also a square, which contradicts (ii). 
	
	Finally, suppose $A$ is anti-diagonal as above and assume $x$ interchanges the $1$-spaces in $\beta$. Here we get 
		\begin{equation}\label{e:c24}
	-a^2\mu^{-i}=b^{-p^j}c^{-1}=c^{-p^j}b^{-1},
	\end{equation}
so $b^{p^j-1}=c^{p^j-1}$ and the condition in (iii) implies that $j=0$ and $i=0$. Therefore $-bc^{-1} = (ab)^2$, which is incompatible with (ii).
	
We conclude that if the scalars $b$ and $c$ satisfy the conditions in (i), (ii) and (iii), then $\{\a,\b\}$ is a base. 
\end{proof}

\begin{lem}
	\label{l:c2_cond_PSigmaL}
	Let $G=\PSigmaL_2(q)$ with $q$ odd and set $\alpha$ and $\beta$ as in Lemma \ref{l:c2_cond}. Then $\{\alpha,\beta\}$ is a base for $G$ if and only if the scalars $b$ and $c$ satisfy the conditions {\rm(i)--(iii)} in Lemma \ref{l:c2_cond}.
\end{lem}

\begin{proof}
By Lemma \ref{l:c2_cond}, it suffices to show that if any of the conditions in (i), (ii) or (iii) fail to hold, then there exists an element $x \ne \pm I_2$ in ${\rm \Sigma L}_{2}(q) = \la \SL_2(q),\phi\ra$ that fixes $\a$ and $\b$. We proceed by inspecting the proof of Lemma \ref{l:c2_cond}, noting that $i=0$ in each of the equations \eqref{e:c21}--\eqref{e:c24}.
	
	As explained in the discussion preceding Lemma \ref{l:c2_cond}, if $bc=0$ then $\{\a,\b\}$ is not a base. Next assume $-bc^{-1}$ is a square in $\mathbb{F}_q$, say $d^2=-bc^{-1}$. Then setting $a=db^{-1}$ gives $-a^2 = b^{-1}c^{-1}$ and we get a solution to \eqref{e:c24} with $j=0$. 
	Finally, suppose $b^{p^k-1} = c^{p^k-1}$ for some $0<k<f$ and choose $a \in \mathbb{F}_q$ with $a^2 = b^{p^k-1}$. Then \eqref{e:c21} is satisfied and we conclude that $x = {\rm diag}(a,a^{-1})\phi^k$ fixes $\a$ and $\b$.
\end{proof}

Let us record three corollaries of Lemma \ref{l:c2_cond_PSigmaL}. The first result allows us to reduce our main problems to the special case $G = \PSigmaL_2(q)$.

\begin{cor}
	\label{c:c2_min}
	Suppose $G\cap \PGL_2(q)=G_0$ and $q$ is odd. Then the Saxl graph $\Sigma(G)$ contains $\Sigma(\PSigmaL_2(q))$ as a subgraph.
\end{cor}

\begin{proof}
Let $\{\a,\b\}$ be a base for $\PSigmaL_2(q)$ with $\a = \{\la e_1\ra, \la e_2\ra\}$ as usual. As explained in the discussion preceding Lemma \ref{l:c2_cond}, we have $\b = \{\la e_1+be_2\ra, \la e_1+ce_2\ra\}$ for nonzero scalars $b$ and $c$, which must satisfy the conditions in parts (i), (ii) and (iii) of Lemma \ref{l:c2_cond} (see Lemma \ref{l:c2_cond_PSigmaL}). Then Lemma \ref{l:c2_cond} implies that $\{\a,\b\}$ is a base for $G$ and the result follows.
\end{proof}

\begin{cor}\label{c:c2_cond}
Let $G=\PSigmaL_2(q)$ with $q$ odd and set 
\[
\beta = \{ \la e_1+be_2 \ra, \la e_1+ce_2\ra\},\;\; \gamma = \{ \la e_1+b'e_2\ra, \la e_1+c'e_2 \ra\}
\]
where $b,c,b',c'$ are nonzero scalars with $b \ne c$ and $b' \ne c'$. Then $\{\beta,\gamma\}$ is a base for $G$ if and only if
\[
	\{b',c'\}=\left\{\frac{b(c-b)+dc}{c-b+d}, \frac{b(c-b)+ec}{c-b+e}\right\}
\]
for scalars $d,e \in \mathbb{F}_q$ with $d,e \ne b-c$ satisfying conditions {\rm(i)--(iii)} in Lemma \ref{l:c2_cond}.
\end{cor}

\begin{proof}
Since $G$ acts primitively on $\O$, it follows that the normal subgroup $G_0$ is transitive. Therefore, $\beta=\alpha^g$ for some $g\in G_0$ and we note that $g$ maps the set of neighbours of $\alpha$ in $\Sigma(G)$ to the set of neighbours of $\beta$. More precisely, we can take $g$ to be the image of the matrix 
\[
\begin{pmatrix}
1&(c-b)^{-1}\\
b&c(c-b)^{-1}
\end{pmatrix} \in {\rm SL}_{2}(q).
\]

Suppose $\{\b,\gamma\}$ is a base, so $\gamma = \delta^g$ for some neighbour $\delta$ of $\a$. By Lemma \ref{l:c2_cond_PSigmaL}, we have
$\delta = \{\la e_1+de_2\ra, \la e_1+ee_2\ra\}$ for scalars $d,e \in \mathbb{F}_q$ satisfying the conditions in (i)--(iii) of Lemma \ref{l:c2_cond} and by applying $g$ we get 
\[
\gamma = \{\la (1+d(c-b)^{-1})e_1+(b+dc(c-b)^{-1})e_2\ra, \la (1+e(c-b)^{-1})e_1+(b+ec(c-b)^{-1})e_2\ra\}.
\]
Here the coefficients $1+d(c-b)^{-1}$ and $1+e(c-b)^{-1}$ are nonzero, so $d,e \ne b-c$ and we deduce that $\gamma$ has the required form. 

Conversely, if $\gamma$ has the given form then $\gamma = \delta^g$ for some $\delta \in \O$ with $\{\a,\delta\}$ a base and it follows that $\{\b,\gamma\}$ is a base.
\end{proof}

\begin{cor}\label{c:c2_val}
Let $G=\PSigmaL_2(q)$ with $q$ odd and let $m$ be the number of non-squares in $\mathbb{F}_q$ that are not contained in any proper subfield of $\mathbb{F}_q$. Then $\Sigma(G)$ has valency $m(q-1)/2$ and thus $G$ has exactly $m/2f$ regular suborbits on $\O$.
\end{cor}

\begin{proof}
We consider the neighbours of $\a = \{ \la e_1 \ra, \la e_2 \ra\}$. Suppose $\beta=\{\la e_1+b_0e_2\ra,\la e_1+e_2\ra\}$. By Lemma \ref{l:c2_cond_PSigmaL}, $\{\a,\beta\}$ is a base if and only if $-b_0$ is a non-square that is not contained in any proper subfield of $\mathbb{F}_q$. Therefore, there are $m$ choices for $\b$. More generally, if $\beta=\{\la e_1+be_2\ra,\la e_1+ce_2\ra\}$ with $c \ne 0$, then $\{\alpha,\beta\}$ is a base if and only if $b=cb_0$ for some $b_0$ as above. Since there are $q-1$ choices for $c$ and we can interchange the two spaces comprising $\b$, we conclude that $\Sigma(G)$ has valency $m(q-1)/2$. Since $|H| = (q-1)f$, it follows that $G$ has precisely $m/2f$ regular suborbits on $\O$.
\end{proof}

We are now in a position to prove our first main result for the $\C_2$-actions of groups with socle $G_0 = {\rm L}_{2}(q)$. The following proposition extends Theorem \ref{t:cd} by establishing the main conjecture of \cite{SaxlGraph} for these groups.

\begin{prop}\label{p:psl2_c2star}
Property \eqref{e:star} holds if $G_0 = {\rm L}_{2}(q)$ and $H$ is of type ${\rm GL}_{1}(q) \wr S_2$.
\end{prop}

\begin{proof}
We may assume $q>27$. Recall that $b(G)=2$ if and only if $G$ does not contain $\PGL_2(q)$ as a proper subgroup. If $G=\PGL_2(q)$, then  $\Sigma(G)$ is isomorphic to the Johnson graph $J(q+1,2)$ and we immediately deduce that \eqref{e:star} holds (as noted in \cite[Example 3.9]{SaxlGraph}).
	
For the remainder, we may assume that $G\cap \PGL_2(q)=G_0$ and $q$ is odd. In view of Corollary \ref{c:c2_min}, we only need to consider the group $G = \PSigmaL_2(q)$. Fix $\alpha=\{\la e_1\ra, \la e_2\ra \}$ as before. By Theorem \ref{t:cd}, it suffices to show that if $\{\alpha,\beta\}$ is a base, then there exists $\gamma \in \O$ such that both $\{\alpha,\gamma\}$ and $\{\beta,\gamma\}$ are bases. 

By Lemma \ref{l:c2_cond_PSigmaL} we have $\beta=\{\la e_1+be_2\ra,\la e_1+ce_2\ra\}$, where $b,c \in \mathbb{F}_q$ are nonzero scalars such that $-bc^{-1}$ is a non-square and is not contained in any proper subfield of $\mathbb{F}_q$. Set $\gamma=\{\la e_1-be_2\ra,\la e_1-ce_2\ra \} \in \O$ and note that $\{\a,\gamma\}$ is a base by Lemma \ref{l:c2_cond}. By Corollary \ref{c:c2_cond}, $\{\beta,\gamma\}$ is a base if and only if there exists $d,e \in \mathbb{F}_q$ with $d,e \ne b-c$ such that 
\begin{equation}\label{e:uw}
	\{-b,-c\}=\left\{\frac{b(c-b)+dc}{c-b+d},\frac{b(c-b)+ec}{c-b+e}\right\}
\end{equation}
and $d,e$ satisfy the conditions in parts (i), (ii) and (iii) in Lemma \ref{l:c2_cond}. 
	
Set $d=\frac{2b(b-c)}{b+c}$ and $e=\frac{b^2-c^2}{2c}$. Then $d,e \ne b-c$, \eqref{e:uw} holds and $de \ne 0$. In addition,
	\[
	-de^{-1}=-bc^{-1}\left(\frac{2c}{b+c}\right)^2=-4(bc^{-1}+cb^{-1}+2)^{-1}
	\] 
	and we immediately deduce that $-de^{-1}$ is a non-square in $\mathbb{F}_q$. 

Finally, we claim that $de^{-1}$ is not contained in a proper subfield of $\mathbb{F}_q$. To do this, it suffices to show that $\eta=bc^{-1}+cb^{-1}$ is not contained in such a subfield. With this aim in mind, it will be useful to observe that 
\begin{align*}
\eta^{p^k}-\eta & =(bc^{-1})^{p^k}+(bc^{-1})^{-p^k}-bc^{-1}-(bc^{-1})^{-1} \\
& = (bc^{-1})^{-p^k}((bc^{-1})^{p^k+1}-1)((bc^{-1})^{p^k-1}-1)
\end{align*}
for $1 \leqs k < f$, so $\eta$ is contained in the subfield $\mathbb{F}_{p^k}$ of $\mathbb{F}_q$ if and only if this expression is $0$. Now since $b$ and $c$ satisfy the condition in part (iii) of Lemma \ref{l:c2_cond}, it follows that $(bc^{-1})^{p^k-1}-1 \ne 0$, whence $\eta \in \mathbb{F}_{p^k}$ if and only if $(bc^{-1})^{p^k+1}=1$. If the latter equality holds, then $bc^{-1}\in\mathbb{F}_{p^{2k}}$ and thus $2k=f$. In particular, this implies that both $-1$ and $bc^{-1}$ are squares, which contradicts (ii) in Lemma \ref{l:c2_cond}. 

This justifies the claim and we conclude that $d$ and $e$ satisfy the conditions in parts (i), (ii) and (iii) of Lemma \ref{l:c2_cond}. In particular, $\{\beta,\gamma\}$ is a base and the result follows.
\end{proof}

Next we turn to the problem of determining when $G$ has a unique regular suborbit on $\O$. We will need the following number-theoretic result, where $\phi$ and $\gamma= 0.57721...$ denote Euler's totient function and Euler's constant, respectively.

\begin{lem}\label{l:euler}
For every integer $n \geqs 3$,
\[
\phi(n) > \frac{n}{e^{\gamma}\log\log n+\frac{3}{\log\log n}}.
\]
\end{lem}

\begin{proof}
See \cite[Theorem 15]{RS}.
\end{proof}

\begin{prop}\label{p:psl2_c2reg}
Suppose $G_0 = {\rm L}_{2}(q)$ and $H$ is of type ${\rm GL}_{1}(q) \wr S_2$. Then $G$ has a unique regular suborbit if and only if $G = {\rm PGL}_{2}(q)$ and $q \geqs 4$, $q \ne 5$.
\end{prop}

\begin{proof}
In view of Proposition \ref{p:comp}, we may assume $q>27$ and we recall that $G = {\rm PGL}_2(q)$ has a unique regular suborbit. For the remainder we may assume $G \cap {\rm PGL}_{2}(q) = G_0$ and our aim is to show that $G$ has at least two regular suborbits. By Corollary \ref{c:c2_min}, we may assume that $G = \PSigmaL_2(q)$, in which case $G$ has $m/2f$ regular suborbits by Corollary \ref{c:c2_val}, where $m$ is the number of non-squares in $\mathbb{F}_q$ that are not contained in any proper subfield of $\mathbb{F}_q$. Any primitive element of $\mathbb{F}_q$ has this property and there are $\phi(q-1)$ such elements in $\mathbb{F}_q$. By applying the lower bound in Lemma \ref{l:euler} we deduce that $\phi(q-1) \geqs 4f$ for all $q>27$ and the result follows.
\end{proof}

Finally, we turn to the clique number of $\Sigma(G)$. If $G = {\rm PGL}_{2}(q)$ then $\omega(G) = q$ since $\Sigma(G)$ is isomorphic to the Johnson graph $J(q+1,2)$, whence $\omega(G) \geqs q$ if $G = G_0$. For $5<q<1000$ we have used the computational approach described in Section \ref{ss:clique} to verify the bound $\omega({\rm P\Sigma L}_{2}(q)) \geqs 5$ (in view of Corollary \ref{c:c2_min}, this implies that $\omega(G) \geqs 5$ whenever $5<q<1000$ is odd and $G \cap {\rm PGL}_{2}(q) = G_0$). We expect $\omega(G) \geqs 5$ for all $q>5$, but we have not been able to prove this.

\begin{rem}\label{r:c2_clique}
Let us say more about the difficulties that arise when trying to construct a clique of size $5$ when $G = \PSigmaL_2(q)$. By Proposition \ref{p:psl2_c2star} we see that $\omega(G)\geqs 3$. More precisely, $\{\alpha,\beta,\gamma\}$ is a clique of size $3$, where 
\[
\alpha = \{\la e_1\ra,\la e_2\ra\}, \; \beta = \{\la e_1+be_2\ra,\la e_1+ce_2\ra\}, \; \gamma = \{\la e_1-be_2\ra,\la e_1-ce_2\ra\}
\]
and $b,c \in \mathbb{F}_{q}^{\times}$ satisfy the conditions in parts (i), (ii) and (iii) of Lemma \ref{l:c2_cond}. Similarly, if we choose different scalars $b',c' \in \mathbb{F}_q^{\times}$ then we can construct another clique $\{\a,\b',\gamma'\}$, where $\beta' = \{\la e_1+b'e_2\ra,\la e_1+c'e_2\ra\}$ and $\gamma' = \{\la e_1-b'e_2\ra,\la e_1-c'e_2\ra\}$. Then 
$\{\a,\b,\b',\gamma,\gamma'\}$ is a clique of size $5$ if $\{\beta, \beta'\}$ and $\{\beta, \gamma'\}$ are bases. So in view of Corollary \ref{c:c2_cond}, we need to find scalars $d,e \ne b-c$ satisfying conditions (i)--(iii) in Lemma \ref{l:c2_cond} such that 
	\[
	\{b',c'\}=\left\{\frac{b(c-b)+dc}{c-b+d}, \frac{b(c-b)+ec}{c-b+e}\right\},
	\] 
	together with another pair of scalars $d',e' \ne b-c$ satisfying the same conditions with 
	\[
	\{-b',-c'\}=\left\{\frac{b(c-b)+d'c}{c-b+d'}, \frac{b(c-b)+e'c}{c-b+e'}\right\}.
	\] 
Here the main difficulty arises in verifying the required conditions in Lemma \ref{l:c2_cond}. For example, if we fix $b'$ and $c'$, then we need to show that $de^{-1}$ is not contained in a proper subfield of $\mathbb{F}_q$, which is not an easy condition to check. At the same time, we also need to verify the corresponding condition for $d'e'^{-1}$, which is an additional complication.
\end{rem}

\subsection{$\C_3$-actions}\label{ss:psl2_c3}

In this section we assume $H$ is of type ${\rm GL}_{1}(q^2)$, so $H_0 = D_{2(q+1)/h}$ and $|\O| = \frac{1}{2}q(q-1)$, where $h=(2,q-1)$. By Proposition \ref{p:comp} we may assume $q>27$ and we note that \cite[Lemma 4.8]{Burness2020base} gives $b(G) \leqs 3$, with equality if and only if ${\rm PGL}_{2}(q) \leqs G$. Therefore, we may assume 
$q$ is odd and $G\cap \PGL_2(q)=G_0$, so either 
\begin{itemize}\addtolength{\itemsep}{0.2\baselineskip}
\item[{\rm (a)}] $G=\langle G_0,\phi^j\rangle$ for some $j$ in the range $0\leqs j<f$; or
\item[{\rm (b)}] $G=\langle G_0,\delta\phi^j\rangle$ with $0<j<f$ and $f/(f,j)$ even.
\end{itemize}

\vspace{2mm}

Following \cite{Burness2020base}, it will be helpful to identify $G_0$ with the unitary group $X_0 = \UU_2(q)$ and $H$ with a maximal subgroup of type $\GU_1(q)\wr S_2$. We may then identify $\O$ with the set of orthogonal pairs of nondegenerate $1$-dimensional subspaces of the natural module $U$ for $X_0$, which is defined over $\mathbb{F}_{q^2}$. As in the proof of \cite[Lemma 4.8]{Burness2020base}, fix an orthonormal basis $\{u,v\}$ for $U$ and set $\a = \{\la u \ra, \la v \ra \} \in \O$. For each nonzero scalar $b \in \mathbb{F}_{q^2}$ with $b^{q+1}\ne -1$ we define
\[
\omega_b=\{\la u+bv\ra,\la u-b^{-q}v\ra\} \in \O.
\]
Then 
\[
\O = \{\a\} \cup \{ \omega_{b} \,:\, b \in \mathbb{F}_{q^2}^{\times},\, b^{q+1} \ne -1\}
\]
and we note that $\omega_{b} = \omega_{-b^{-q}}$. We will abuse notation by writing $\phi$ for the field automorphism of $X_0$ that corresponds to the map $\eta \mapsto \eta^{p}$ on $\mathbb{F}_{q^2}$ and we will assume that  
\[
(au+bv)^{\phi} = a^pu + b^pv
\]
for all $a,b \in \mathbb{F}_{q^2}$. We define ${\rm \Sigma U}_2(q) = \la {\rm SU}_{2}(q),\phi\ra$ and ${\rm P\Sigma U}_{2}(q) = \la X_0, \phi \ra = X_0.f$, noting that $X_0 \cap \la \phi \ra = \la \phi^f \ra$. In this setting, the two cases we need to consider are as described in (a) and (b) above, with $G_0$ replaced by $X_0$. Note that in (b),  the diagonal automorphism $\delta$ is the image of a diagonal matrix ${\rm diag}(\l^{q-1},1) \in {\rm GU}_{2}(q)$ with respect to the basis $\{u,v\}$ for $U$, where $\mathbb{F}_{q^2}^{\times} = \la \l \ra$. 

We begin with the following result, which is the analogue of Lemma \ref{l:c2_cond} for the $\C_3$-actions we are considering here. Note that the sufficient condition in the lemma is equivalent to the non-containment of $b^{\frac{1}{2}(q+1)}$ in a proper subfield of $\mathbb{F}_{q^2}$.

\begin{lem}\label{l:c3_cond}
	Suppose $G\cap \PGL_2(q)=G_0$ with $q$ odd. Then $\{\alpha,\omega_b\}$ is a base for $G$ if 
	\begin{equation}\label{e:bcon}
	b^{\frac{1}{2}(q+1)(p^k-1)} \ne 1 
	\end{equation}
	for all $0<k<2f$.
\end{lem}

\begin{proof}
Suppose $b$ satisfies the condition in \eqref{e:bcon} for all $0<k<2f$ and let us assume 
\[
x = AB^i\phi^j \in \la \GU_2(q),\phi\ra
\]
fixes $\a$ and $\omega_b$, where $A \in {\rm SU}_{2}(q)$, $B={\rm diag}(\l^{q-1},1)$, $0\leqs i<q+1$, $0 \leqs j < 2f$ with $j \ne f$, and $j>0$ if $i>0$. In order to prove that $\{\a,\omega_b\}$ is a base for $G$, it suffices to show that if $x$ fixes $\a$ and $\omega_b$, then $i=j = 0$ and $A = \pm I_2$. So let us assume $x$ fixes $\a$ and $\omega_b$, which means that  $A$ is either diagonal or anti-diagonal with respect to the basis $\{u,v\}$ for the natural ${\rm SU}_{2}(q)$-module $U$.
	
	First assume $A={\rm diag}(a,a^{-1})$ is diagonal, so $a^{q+1}=1$. If $x$ fixes the two spaces in $\omega_b$, then 
	\begin{align*}
	(u+bv)^x & = a\lambda^{i(q-1)}u+a^{-1}b^{p^j}v=\eta_1(u+bv)\\
	(u-b^{-q}v)^x & = a\lambda^{i(q-1)}u-a^{-1}b^{-qp^j}v=\eta_2(u-b^{-q}v)
	\end{align*}
	for some $\eta_1,\eta_2\in\mathbb{F}_{q^2}^{\times}$, whence 
	\begin{equation}\label{e:c31}
	a^2\lambda^{i(q-1)}=b^{p^j-1}=b^{q(1-p^j)}.
	\end{equation}
	Since $a^{q+1}=1$ we get $(b^{q+1})^{(p^j-1)}=1$, which implies $(b^{q+1})^{\frac{1}{2}(p^{2j}-1)}= 1$ and thus $2j=0$ or $2f$ are the only possibilities. But we are assuming $j\ne f$, hence $j=0$ and thus $i=0$. Therefore, \eqref{e:c31} gives $a^2=1$ and we conclude that $x=\pm I_2$.
	
	Similarly, if $x$ interchanges the spaces in $\omega_b$, then
	\begin{align*}
	(u+bv)^x & = a\lambda^{i(q-1)}u+a^{-1}b^{p^j}v = \eta_1(u-b^{-q}v) \\
	(u-b^{-q}v)^x & = a\lambda^{i(q-1)}u-a^{-1}b^{-qp^j}v=\eta_2(u+bv)
	\end{align*}
	for some $\eta_1,\eta_2\in\mathbb{F}_{q^2}^{\times}$ and we deduce that 
	\begin{equation}\label{e:c32}
		-a^2\lambda^{i(q-1)}=b^{p^j+q}=b^{-qp^j-1}.
	\end{equation}
	In particular, since $a^{q+1}=1$, it follows that 
	\[
	(b^{q+1})^{\frac{1}{2}(p^{j+f}+1)}=(-1)^{\frac{1}{2}(q+1)}\lambda^{\frac{1}{2}i(q^2-1)}=(-\lambda^{i(q-1)})^{\frac{1}{2}(q+1)}
	\]
	and thus
	\[
	(b^{q+1})^{\frac{1}{2}(p^{2j}-1)}=(b^{q+1})^{\frac{1}{2}(p^{2(j+f)}-1)}=(-\lambda^{i(q-1)})^{\frac{1}{2}(q+1)(p^{j+f}-1)}=1.
	\]
	Since \eqref{e:bcon} holds and $j \ne f$ we deduce that $j=0$ is the only possibility, implying $i=0$. Then \eqref{e:c32} gives $b^{q+1}=b^{-q-1}$ and thus $b^{q+1}=\pm 1$. By construction we have $b^{q+1} \ne -1$ (since $\omega_b \in \O$), while \eqref{e:bcon} implies that $b^{q+1} \ne 1$. Therefore, we have reached a contradiction and this case does not arise. 
	
	Now let us assume $x$ interchanges the two spaces in $\a$, so 
	\[
	A=\begin{pmatrix}
	0&a\\
	-a^{-1}&0
	\end{pmatrix}
	\]
	is anti-diagonal and $a^{q+1}=1$. If $x$ fixes the spaces in $\omega_b$, then  
	\begin{align*}
	(u+bv)^x & = ab^{p^j}u-a^{-1}\lambda^{i(q-1)}v = \eta_1(u+bv)\\
	(u-b^{-q}v)^x & = -ab^{-qp^j}u-a^{-1}\lambda^{i(q-1)}v = \eta_2(u-b^{-q}v)
	\end{align*}
	for some $\eta_1,\eta_2\in\mathbb{F}_{q^2}^{\times}$ and we get  
	\begin{equation}\label{e:c33}
	-a^2\lambda^{-i(q-1)}=b^{-p^j-1}=b^{q+qp^j}.
	\end{equation}
	This implies that $(b^{q+1})^{\frac{1}{2}(p^{2j}-1)}=1$, which leads to a contradiction as above. Finally, suppose $x$ interchanges the two spaces in $\omega_b$. Here we get  
	\begin{equation}\label{e:c34}
	a^2\lambda^{i(q-1)}=b^{q-p^j}=b^{qp^j-1}
	\end{equation}
	and thus $(b^{q+1})^{\frac{1}{2}(p^{f+j}-1)}=\lambda^{\frac{1}{2}i(q^2-1)}=\pm 1$ and so $(b^{q+1})^{\frac{1}{2}(p^{2(f+j)}-1)}=1$. It follows that $j=0$ and $i=0$, so $(b^{q+1})^{\frac{1}{2}(p^f-1)}=1$ and this is incompatible with \eqref{e:bcon}. 
\end{proof}

\begin{lem}\label{l:c3_cond_PSigmaL}
	Let $G=\PSigmaL_2(q)$ with $q$ odd. Then $\{\alpha,\omega_b\}$ is a base for $G$ if and only if (\ref{e:bcon}) holds for all $0<k<2f$.
\end{lem}

\begin{proof}
	By Lemma \ref{l:c3_cond}, it suffices to show that if the condition in \eqref{e:bcon} fails to hold, then $\{\alpha,\omega_b\}$ is not a base for $G$. 
	So let us assume $k$ is an integer such that $0<k<2f$ and 
	\[
	b^{\frac{1}{2}(q+1)(p^k-1)}=1.
	\]
	If $k \ne f$ then by setting $j=k$ we deduce that \eqref{e:c31} holds with $a = b^{(p^k-1)/2}$ and $i=0$, otherwise \eqref{e:c34} holds with $a = b^{(q-1)/2}$ and $i=j = 0$. In both cases we conclude that 
	$\{\a,\omega_b\}$ is not a base and the result follows.
\end{proof}

\begin{rem}\label{r:c3g0}
By inspecting the proofs of Lemmas \ref{l:c3_cond} and \ref{l:c3_cond_PSigmaL}, we deduce that if $G = G_0$ and $q$ is odd, then $\{\a,\omega_b\}$ is a base for $G$ if and only if $b$ is a non-square in $\mathbb{F}_{q^2}$. The same criterion was established in the proof of \cite[Theorem 10]{BurnessHarper} for $q \equiv 3 \imod{4}$ and we note that a very similar argument can be used to reach the same conclusion when $q \equiv 1 \imod{4}$. In addition, we refer the reader to \cite[Lemma 7.9]{BurnessHarper} for a complete list of the subdegrees of $G=G_0$ when $q \geqs 11$ is odd.
\end{rem}

We can now reduce our main problems to the special case $G = \PSigmaL_2(q)$.

\begin{cor}\label{c:c3_min}
	Suppose $G\cap \PGL_2(q) = G_0$ and $q$ is odd. Then the Saxl graph $\Sigma(G)$ contains $\Sigma(\PSigmaL_2(q))$ as a subgraph.
\end{cor}

\begin{proof}
	If $\{\alpha,\omega_b\}$ is a base for $\PSigmaL_2(q)$, then Lemma \ref{l:c3_cond_PSigmaL} implies that \eqref{e:bcon} holds and thus $\{\alpha,\omega_b\}$ is a base for $G$ by Lemma \ref{l:c3_cond}.
\end{proof}

The following technical result is a key observation. 

\begin{cor}\label{c:c3_cond}
	Let $G=\PSigmaL_2(q)$ with $q$ odd and let $b,c \in \mathbb{F}_{q^2}^{\times}$ be scalars such that $b^{q+1}, c^{q+1} \ne -1$ and $c \not\in \{b,-b^{-q}\}$. Then $\{\omega_b,\omega_{c}\}$ is a base for $G$ if and only if 
\[
\frac{ba_1^{-2}(b+b^{-q})+b^{-q}d}{a_1^{-2}(b+b^{-q})-d} \in \{c,-c^{-q}\}
\]
for scalars $a_1,d \in \mathbb{F}_{q^2}$ satisfying all of the following conditions:
	\begin{itemize}\addtolength{\itemsep}{0.2\baselineskip}
\item[{\rm (i)}] $a_1^{q+1} = 1+b^{q+1}$;	
\item[{\rm (ii)}] $d^{q+1}\ne -1$ and $d \not\in \{a_1^{-2}(b+b^{-q}),-b^{q+1}a_1^{-2}(b+b^{-q})\}$;
\item[{\rm (iii)}] $d^{\frac{1}{2}(q+1)(p^k-1)} \ne 1$ for all $0<k<2f$.
\end{itemize}
\end{cor}

\begin{proof}
Choose $a_1 \in \mathbb{F}_{q^2}$ such that $a_1^{q+1} = 1+b^{q+1}$ and set $a_2 = -(b+b^{-q})a_1^{-1}$ (note that $a_1$ exists since $1+b^{q+1} \in \mathbb{F}_q$). Then $a_2^{q+1} = 1+b^{-(q+1)}$ and we deduce that both $a_1^{-1}(u+bv)$ and $a_2^{-1}(u-b^{-q}v)$ are unit vectors. Let $g \in X_0 = {\rm U}_2(q)$ be the image of the matrix
	\[
	A=
	\begin{pmatrix}
	a_1^{-1}&a_2^{-1}\\
	ba_1^{-1}&-b^{-q}a_2^{-1}
	\end{pmatrix} \in {\rm SU}_{2}(q),
	\]
	which is expressed in terms of the basis $\{u,v\}$ for $U$. Note that 
	$\alpha^g = \omega_b$.
	
	First assume $\{\omega_b,\omega_{c}\}$ is a base for $G$, so $\omega_{c} = \omega_d^g$ for some neighbour $\omega_d$ of $\alpha$ in $\Sigma(G)$. Then $d^{q+1} \ne -1$ and Lemma \ref{l:c3_cond_PSigmaL} implies that $d$ satisfies the condition in (iii). By applying $g$ we get
	\[
	(u+dv)^g = (a_1^{-1}+a_2^{-1}d)u+(ba_1^{-1}-b^{-q}a_2^{-1}d)v.
	\]
	Since $\omega_c\ne \alpha$, the coefficients of $u$ and $v$ in this expression are nonzero and we deduce that $d$ satisfies the remaining conditions in (ii). In particular,  
	\begin{align*}
	\langle u+dv\rangle^g & = \left\la u+\frac{ba_1^{-1}a_2-b^{-q}d}{a_1^{-1}a_2+d}v\right\ra
	= \left\la u+\frac{ba_1^{-2}(b+b^{-q})+b^{-q}d}{a_1^{-2}(b+b^{-q})-d}v\right\ra \\
	\langle u-d^{-q}v\rangle ^g & = \left\la u+\frac{ba_1^{-2}(b+b^{-q})-b^{-q}d^{-q}}{a_1^{-2}(b+b^{-q})+d^{-q}}v\right\ra
	\end{align*}
	and we conclude that if $\{\omega_b,\omega_c\}$ is a base for $G$ then all of the required conditions are satisfied.
		
	Conversely, if $c$ has the given form for scalars $a_1$ and $d$ satisfying all of the given conditions, then $\{\alpha,\omega_d\}$ is a base for $G$ (via the condition in (iii)) and $\omega_c = \omega_d^g$ for some $g \in G_0$ with $\omega_b = \a^g$. Therefore, $\{\omega_b,\omega_c\}$ is also a base for $G$.
\end{proof}

We are now in a position to extend Theorem \ref{t:cd} by establishing \eqref{e:star} for the $\C_3$-actions of groups with socle ${\rm L}_2(q)$.

\begin{prop}\label{p:psl2_c3star}
Property \eqref{e:star} holds if $G_0 = {\rm L}_{2}(q)$ and $H$ is of type ${\rm GL}_{1}(q^2)$.
\end{prop}

\begin{proof}
We may assume $q>27$ (see Proposition \ref{p:comp}) and we recall that $b(G)=2$ if and only if $q$ is odd and $G\cap \PGL_2(q) = G_0$. In view of Corollary \ref{c:c3_min}, we may assume that $G = \PSigmaL_2(q)$. By Theorem \ref{t:cd}, it suffices to show that if $\{\a,\omega_b\}$ is a base for $G$, then there exists $c \in \mathbb{F}_{q^2}^{\times}$ with $c^{q+1}\ne -1$ such that both $\{\a,\omega_{c}\}$ and $\{\omega_b,\omega_{c}\}$ are also bases. Note that $b^{q+1} \ne -1$ and $b$ satisfies the condition in \eqref{e:bcon} for all $0<k<2f$ (see Lemma \ref{l:c3_cond_PSigmaL}).

We claim that all of the above properties hold with $c = -b$. Clearly, we have $c^{q+1} = b^{q+1} \ne -1$ and $c^{\frac{1}{2}(q+1)(p^k-1)} \ne 1$ for all $0<k<2f$, so $\{\alpha,\omega_{c}\}$ is a base. It remains to prove that $\{\omega_b,\omega_{c}\}$ is a base.

As in Corollary \ref{c:c3_cond}, fix a scalar $a_1 \in \mathbb{F}_{q^2}^{\times}$ such that $a_1^{q+1} = 1+b^{q+1}$ and set 
\[
d = \frac{2ba_1^{-2}(b+b^{-q})}{b-b^{-q}} \in \mathbb{F}_{q^2}^{\times}.
\]
Then 
\[
c = \frac{ba_1^{-2}(b+b^{-q})+b^{-q}d}{a_1^{-2}(b+b^{-q})-d}
\]
and it remains to show that $d$ satisfies all the conditions in parts (ii) and (iii) of Corollary \ref{c:c3_cond}.

If $d\in \{a_1^{-2}(b+b^{-q}),-b^{q+1}a_1^{-2}(b+b^{-q})\}$ then $b^{q+1}=-1$, which is a contradiction. In addition, if we write $\alpha^g = \omega_b$ as in the proof of Corollary \ref{c:c3_cond}, then $\la u+dv\ra = \la u+cv\ra^{g^{-1}}$ and $\la u-d^{-q}v\ra = \la u-c^{-q}v\ra^{g^{-1}}$. Since $c^{q+1}\ne -1$, we have $u+cv\ne u-c^{-q}v$ and thus $u+dv\ne u-d^{-q}v$. Therefore, $d^{q+1}\ne -1$ and we conclude that $d$ satisfies all of the conditions in part (ii) of Corollary \ref{c:c3_cond}.

Finally, we need to show that 
\[
d^{\frac{1}{2}(q+1)} =\pm 2\left(b^{\frac{1}{2}(q+1)} - b^{-\frac{1}{2}(q+1)}\right)^{-1}
\]
is not contained in a proper subfield of $\mathbb{F}_{q^2}$. Set $e = b^{\frac{1}{2}(q+1)}$, which is not in a proper subfield by Lemma \ref{l:c3_cond_PSigmaL}, and note that it suffices to show that $e-e^{-1}$ is also not contained in a proper subfield. Fix an integer $0<k<2f$ and observe that 
\[
(e-e^{-1})^{p^k}-(e-e^{-1})=e^{-p^k}(e^{p^k+1}+1)(e^{p^k-1}-1),
\]
so $e-e^{-1} \in \mathbb{F}_{p^k}$ if and only if this expression is $0$. In view of \eqref{e:bcon}, this holds if and only if $e^{p^k+1} = -1$. So let us assume this relation holds. Then $e^{p^{2k}-1}=(-1)^{p^k-1} = 1$ and thus $2k = 0$ or $2f$. If $2k=2f$ then $k = f$ and so $e^{q+1} = e^{p^k+1} = -1$. This implies that
\[
-1 = e^{q+1} = b^{\frac{1}{2}(q+1)^2} = b^{\frac{1}{2}(q^2+1)}b^q = -b^{q+1}
\] 
and so $b$ is a square, which is incompatible with \eqref{e:bcon}. Therefore $k = 0$, so $e^2 = -1$
and $e-e^{-1} = -2e^{-1}$. The result now follows since Lemma \ref{l:c3_cond_PSigmaL} implies that $e^{-1}$ is not contained in a proper subfield of $\mathbb{F}_{q^2}$.
\end{proof}

\begin{prop}\label{p:psl2_c3reg}
Suppose $G_0 = {\rm L}_{2}(q)$ and $H$ is of type ${\rm GL}_{1}(q^2)$. Then $G$ has a unique regular suborbit if and only if $G= {\rm L}_{2}(5)$ or ${\rm L}_{2}(9).2 = {\rm M}_{10}$.
\end{prop}

\begin{proof}
In view of Proposition \ref{p:comp} and Corollary \ref{c:c3_min}, we may assume $q>27$ is odd and $G=\PSigmaL_2(q)$. Let $r$ be the number of regular suborbits of $G$. If $q$ is a prime then $G = G_0$ and \cite[Lemma 7.9]{BurnessHarper} gives $r=(q-\ell)/4$, where $q \equiv \ell \imod{4}$ with $\ell \in \{1,3\}$. For the remainder, we may assume $q>p$.

Let $\l$ be a primitive element of $\mathbb{F}_{q^2}$. Then by Lemma \ref{l:c3_cond}, we see that $\{\a,\omega_\l\}$ is a base for $G$.
Therefore, the valency of $\Sigma(G)$ is at least $\phi(q^2-1)/2$, where $\phi$ is Euler's function and thus $r \geqs \phi(q^2-1)/2f(q+1)$ since $|H|=f(q+1)$. By applying the lower bound in Lemma \ref{l:euler} we deduce that  
\[
\frac{\phi(q^2-1)}{2f(q+1)} \geqs 2
\]
for $q>27$ and the desired result follows.
\end{proof}

Finally, we consider the clique number $\omega(G)$ of $\Sigma(G)$. Our main result is the following, which establishes a strong form of Theorem \ref{t:main11} when $G = G_0$ is simple. 

\begin{prop}
	\label{p:psl2_c3cliq_G0}
	Suppose $G = \LL_2(q)$ and $H$ is of type $\GL_1(q^2)$, where $q$ is odd. Then $\omega(G)\geqs \frac{1}{2}(q-1)$.
\end{prop}

\begin{proof}
	Let $b$ be a non-square in $\mathbb{F}_{q^2}$ with $b^{q+1}\ne -1$. Then $bx$ is a non-square for all $x \in \mathbb{F}_q$ and thus $\{\alpha,\omega_{bx}\}$ is a base for $G$ if and only if $(bx)^{q+1}\ne -1$ (see Remark \ref{r:c3g0}). We claim that
	\[
	C=\{\a\}\cup \{\omega_{bx} \, : \, x\in\mathbb{F}_q^\times, \, (bx)^{q+1}\ne -1\}
	\]
	is a clique in $\Sigma(G)$ with $|C| \geqs \frac{1}{2}(q-1)$. 
	
	To see this, first note that if $x \in \mathbb{F}_q^\times$ and $(bx)^{q+1} = -1$, then $b^{q+1} = -x^{-2}$ and there are at most two possibilities for $x$. Let us also observe that if $y = -b^{-(q+1)}x^{-1}$ then $by = -(bx)^{-q}$ and thus $\omega_{bx} = \omega_{by}$. This shows that $|C| \geqs 1+\frac{1}{2}(q-3) = \frac{1}{2}(q-1)$. 
	
	To complete the proof of the proposition, it suffices to show that $\{\omega_{bx},\omega_{by}\}$ is a base for all distinct points $\omega_{bx}, \omega_{by}$ in $C$. To this end, set $c = bx$ and note that $c$ is a non-square such that $c^{q+1}\ne -1$ and $by = cyx^{-1}\in c\mathbb{F}_q^\times$. In this way, the problem is reduced to showing that $\{\omega_b,\omega_{by}\}$ is a base for all $y\in\mathbb{F}_q^\times$ with $(by)^{q+1}\ne -1$ and $by \ne -b^{-q}$ (the latter condition forces $\omega_b \ne  \omega_{by}$).
		
	By Corollary \ref{c:c3_cond}, it suffices to show that there exist scalars $a_1,d \in \mathbb{F}_{q^2}^{\times}$ such that 
	\begin{equation}\label{e:bzz}
	by = \frac{ba_1^{-2}(b+b^{-q})+b^{-q}d}{a_1^{-2}(b+b^{-q})-d},
	\end{equation}
	where $a_1^{q+1} = 1+b^{q+1}$ and $d$ satisfies the conditions in parts (ii) and (iii) of the corollary. In fact, in view of 
	Remark \ref{r:c3g0}, we can replace the condition in (iii) by the property that $d$ is a non-square.
	
	Let $a_1 \in \mathbb{F}_{q^2}^{\times}$ be any scalar such that $a_1^{q+1} = 1+b^{q+1}$, noting that $a_1$ exists since $b^{q+1} \in \mathbb{F}_q$. Since $by \ne -b^{-q}$, it follows that $b^{-(q+1)}+y \ne 0$ and we may define
	\[
	d = \frac{ba_1^{-2}(1+b^{-(q+1)})(y-1)}{b^{-(q+1)}+y}.
	\]
	Since $b^{q+1} \ne -1$ we see that $d \ne a_1^{-2}(b+b^{-q})$ and by rearranging we deduce that \eqref{e:bzz} holds. 
	In particular, since $by \ne 0$ we have $d \ne -b^{q+1}a_1^{-2}(b+b^{-q})$. If $d^{q+1}=-1$, then $(y-1)^2 = -b^{q+1}(b^{-(q+1)}+y)^2$, which forces $(b^{q+1}y^2+1)(b^{q+1}+1) = 0$. But this is false because $b^{q+1}y^2 = (by)^{q+1}\ne -1$ and $b^{q+1}\ne -1$, whence  $d^{q+1}\ne -1$ and $d$ satisfies all of the conditions in part (ii) of Corollary \ref{c:c3_cond}.
	Finally, let us observe that $a_1^{-2}$ is a square and $\frac{(1+b^{-(q+1)})(y-1)}{b^{-(q+1)}+y}\in\mathbb{F}_q$, which is also a square. Therefore, $d$ is a non-square since $b$ is a non-square. We conclude that $\{\omega_b,\omega_{by}\}$ is a base for $G$ and this completes the proof of the proposition.
\end{proof}

\begin{cor}
	\label{c:psl2_c3cliq_G0}
	If $G = G_0 = \LL_2(q)$ and $H$ is of type $\GL_1(q^2)$, then either $\omega(G)\geqs 5$, or $q = 5$ and $\omega(G) = 4$.
\end{cor}

\begin{proof}
By Proposition \ref{p:comp} we may assume $q>27$. Now apply Proposition \ref{p:psl2_c3cliq_G0}.
\end{proof}

\begin{rem}\label{r:c3_clique}
As for the $\C_2$-actions from the previous section, we have been unable to extend Corollary \ref{c:psl2_c3cliq_G0} to the base-two groups with $G \ne G_0$. Using a computational approach (see Section \ref{ss:clique}), it is straightforward to show that $\omega({\rm P\Sigma L}_{2}(q)) \geqs 5$ when  $5<q<1000$ and we expect $\omega(G) \geqs 5$ for all $q>5$. 

In order to illustrate some of the difficulties that arise, let us assume $G = \PSigmaL_2(q)$. As explained in the proof of Proposition \ref{p:psl2_c3star}, if $b \in \mathbb{F}_{q^2}$ is a scalar such that $b^{q+1} \ne -1$ and \eqref{e:bcon} holds for all $0<k<2f$, then $\{\alpha,\omega_b,\omega_{-b}\}$ a clique of size $3$. Similarly, we can choose a different scalar $c$ such that $\{\a,\omega_c, \omega_{-c}\}$ is a clique and we can seek to combine these sets to construct a clique $\{\alpha,\omega_b,\omega_{-b},\omega_c,\omega_{-c}\}$ of size $5$. For this approach to work, we need $\{\omega_b,\omega_c\}$ and $\{\omega_b,\omega_{-c}\}$ to be bases, which imposes various conditions on $c$.  For example, we need the scalars
\[
d=\frac{(c-b)a_1^{-2}(b+b^{-q})}{b^{-q}+c},\;\;  e = \frac{-(b+c)a_1^{-2}(b+b^{-q})}{b^{-q}-c}
\]
to satisfy all the conditions in parts (ii) and (iii) of Corollary \ref{c:c3_cond}, where $a_1 \in \mathbb{F}_{q^2}$ is chosen so that $a_1^{q+1} = 1+b^{q+1}$. These conditions are difficult to check and we have been unable to establish the existence of scalars $b$ and $c$ with the desired properties in full generality.
\end{rem}

\subsection{An extension}\label{ss:further}

In this section we take the opportunity to establish \cite[Conjecture 4.5]{SaxlGraph} for all base-two almost simple primitive groups $G$ with socle $G_0 = {\rm L}_{2}(q)$. This extends the main theorem of \cite{ChenDu2020Saxl}, which shows that $\Sigma(G)$ has diameter $2$.

\begin{thm}\label{t:further}
Let $G \leqs {\rm Sym}(\O)$ be a base-two almost simple primitive permutation group with socle $G_0 = {\rm L}_{2}(q)$. Then any two vertices in $\Sigma(G)$ have a common neighbour.
\end{thm}

\begin{proof}
The groups with $q \leqs 11$ can be handled using {\sc Magma} \cite{Magma} (see Section \ref{s:main1} below), so let us assume $q >11$. By combining Proposition \ref{p:bound} with Theorem \ref{t:random}, we may assume that either $G \in \mathcal{L}$ or a point stabiliser $H$ is insoluble. So in view of Propositions \ref{p:psl2_c2star} and \ref{p:psl2_c3star}, we may assume $H$ is insoluble. By inspecting \cite[Tables 8.1 and 8.2]{Low-Dimensional}, one of the following holds:
\begin{itemize}\addtolength{\itemsep}{0.2\baselineskip}
\item[{\rm (a)}] $H = S_5$ or $A_5$, $q \in \{p,p^2\}$ and $p \equiv \pm 1, \pm 3 \imod{10}$; 
\item[{\rm (b)}] $H$ is a subfield subgroup of type ${\rm GL}_{2}(q_0)$, where $q = q_0^k$, $k$ is a prime and $q_0 \geqs 4$.
\end{itemize}

First consider case (a) and recall that we write $i_m(H)$ for the number of elements of order $m$ in $H$. In view of Proposition \ref{p:bound}, it suffices to show that $\what{Q}(G)<1/2$ (see Section \ref{ss:prob}). We refer the reader to \cite[Section 3.2]{BG_book} for detailed information on the conjugacy classes of elements of prime order in $G$. 

Let $x \in H$ be an element of prime order $m$ and note that $m \in \{2,3,5\}$. If $m=2$ then $|x^G| \geqs \frac{1}{2}q^{1/2}(q+1) = b_1$ (minimal if $x$ is an involutory field automorphism) and we note that $i_2(H) \leqs 25 = a_1$. Similarly, if $m \in \{3,5\}$ then $|x^G| \geqs q(q-1) = b_2$ and $i_3(H)+i_5(H) = 44 = a_2$. Therefore, Lemma \ref{l:calc} implies that
\begin{equation}\label{e:qq}
\what{Q}(G) \leqs a_1^2/b_1+a_2^2/b_2
\end{equation}
and we deduce that $\what{Q}(G)<1/2$ if $q \geqs 197$. The remaining cases with $q<197$ can be verified using {\sc Magma}.

Finally, let us consider the subfield subgroups in (b). First assume $k=2$. We claim that $b(G) \geqs 3$. If $p$ is odd then
\[
|\O| = |G_0: {\rm PGL}_{2}(q_0)| = \frac{1}{2}q_0(q_0^2+1)
\]
and the claim follows since $b(G) \geqs \log_{|\O|}|G|>2$. Similarly, if $p=2$ and $G \ne G_0$ then $|\O| = q_0(q_0^2+1)$ and
\[
b(G) \geqs \frac{\log |G|}{\log |\O|} \geqs \frac{\log 2q_0^2(q_0^4-1)}{\log q_0(q_0^2+1)} > 2.
\]
Finally, suppose $p=2$ and $G = G_0$. Here $\log_{|\O|}|G|<2$ and the subdegrees of $G$ are recorded in \cite[p.354]{FI}. By inspection, we see that $G$ does not have a regular suborbit and thus $b(G) \geqs 3$.

Next assume $k \geqs 5$ and let $x \in H$ be an element of prime order. If $x$ is an involutory field automorphism of $G_0$, then $|x^G| \geqs \frac{1}{2}q^{1/2}(q+1) = b_1$ and we note that there are at most $a_1 = q_0^{1/2}(q_0+1)$ such elements in $H$. In each of the remaining cases, we have $|x^G| \geqs \frac{1}{2}q(q-1) = b_2$ and we observe that $|H| \leqs q_0(q_0^2-1)\log q = a_2$. By applying Lemma \ref{l:calc} we deduce that \eqref{e:qq} holds and this gives $\what{Q}(G)<1/2$ as required.

To complete the proof, we may assume $k=3$. This case requires a more refined treatment. First let $x \in H \cap {\rm PGL}_2(q)$ be an element of prime order $m$. If $x$ is unipotent (so $m=p$) then $|x^G| \geqs \frac{1}{2}(q^2-1) = b_1$ and we note that there are exactly $a_1 = q_0-1$ such elements in $H$. Similarly, if $x$ is a semisimple involution then $|x^G| \geqs \frac{1}{2}q(q-1) = b_2$ and we have $i_2({\rm PGL}_2(q_0)) = q_0^2 = a_2$. Next suppose $m \ne p$ and $m \geqs 3$, so $m$ divides $q_0^2-1$ and there are $\frac{1}{2}(m-1)$ distinct $G_0$-classes of such elements in $G$ (and the same number of ${\rm L}_{2}(q_0)$-classes in $H \cap {\rm PGL}_{2}(q)$). If $\{x_1, \ldots, x_t\}$ is a set of representatives of the distinct $G$-classes of these elements, then there exist positive integers $k_i$ such that $\sum_i k_i = \frac{1}{2}(m-1)$ and 
\[
|x_i^G \cap H| = k_iq_0(q_0+\e),\;\; |x_i^G| = k_iq_0^3(q_0^3+\e),
\]
where $\e =1$ if $m$ divides $q_0-1$, otherwise $\e=-1$ (here $|x_i^{G_0}| = |x_i^{{\rm PGL}_2(q)}| = q_0^3(q_0^3+\e)$ and $k_i$ denotes the number of distinct $G_0$-classes that are fused under the action of field automorphisms in $G$). Therefore, the contribution to $\what{Q}(G)$ from elements of order $m$ is equal to
\[
\sum_{i=1}^t\frac{(k_iq_0(q_0+\e))^2}{k_iq_0^3(q_0^3+\e)} = \frac{1}{2}(m-1)\cdot\frac{(q_0+\e)^2}{q_0(q_0^3+\e)}.
\]
If $m$ divides $q_0+1$ then $m-1 \leqs q_0$ and there are at most $\log(q_0+1)$ possibilities for $m$, so the total contribution to $\what{Q}(G)$ from these elements is at most 
\[
f_1(q_0) = \log(q_0+1) \cdot \frac{1}{2}q_0 \cdot \frac{(q_0-1)^2}{q_0(q_0^3-1)}.
\]
Similarly, the contribution from the elements with $m$ dividing $q_0-1$ is no more than
\[
f_2(q_0) = \log(q_0-1) \cdot \frac{1}{2}(q_0-2) \cdot \frac{(q_0+1)^2}{q_0(q_0^3+1)}.
\]

Finally, let us assume $x \in G$ is a field automorphism of order $m$. As above, if $m=2$ then $|x^G| \geqs \frac{1}{2}q^{1/2}(q+1) = b_3$ and there are at most $a_3 = q_0^{1/2}(q_0+1)$ of these elements in $H$. Next suppose $m=3$. Here $|x^G| \geqs q_0^2(q_0^4+q_0^2+1) = b_4$ and we may assume $H = C_G(x)$, which implies that $H$ contains at most 
\[
2\left(1+i_3({\rm L}_2(q_0))\right) \leqs 2\left(1+\frac{|{\rm GL}_2(q_0)|}{(q_0-1)^2}\right) = 2q_0(q_0+1)+2 = a_4
\]
such elements. If $m=5$ then $|x^G|>q_0^{36/5} = b_5$ and there are at most $8q_0^{12/5} =  a_5$ of these elements in $H$. Finally, if $m \geqs 7$ then $|x^G|>q_0^{54/7} = b_6$ and we observe that $|H| \leqs q_0(q_0^2-1)\log q = a_6$.

Set $\a=1$ if $q_0 = q_1^2$ for some $q_1$, otherwise $\a=0$. Similarly, let $\b = 1$ if $q_0 = q_1^5$ and $\gamma = 1$ if $q_0 = q_1^m$ for some prime $m \geqs 7$ (otherwise $\b=0$ and $\gamma=0$, respectively). Then by bringing all of the above estimates together, we conclude that
\[
\what{Q}(G) < f_1(q_0)+f_2(q_0) + \left(a_1^2/b_1+a_2^2/b_2+ a_4^2/b_4\right) + \a a_3^2/b_3+\b a_5^2/b_5 + \gamma a_6^2/b_6
\]
and one checks that this upper bound is less than $1/2$ for all $q_0 \geqs 7$. Finally, the cases with $q_0 \in \{4,5\}$ can be checked using {\sc Magma}.
\end{proof}

\section{Proof of Theorem \ref{t:main1}}\label{s:main1}

In this section we complete the proof of Theorem \ref{t:main1}. Let $G \leqs {\rm Sym}(\O)$ be a permutation group in $\mathcal{G}$ with socle $G_0$ and point stabiliser $H$. Recall that our goal is to show that the Saxl graph $\Sigma(G)$ has the following property:
\[
\mbox{\emph{Any two vertices in $\Sigma(G)$ have a common neighbour}} \tag{$\star$}
\]
which immediately implies that $\Sigma(G)$ has diameter $2$. In view of Propositions \ref{p:bound}, \ref{p:psl2_c2star} and \ref{p:psl2_c3star}, we may assume $G \in \mathcal{G}\setminus \mathcal{L}$ and $Q(G) \geqs 1/2$, so the relevant groups can be determined by inspecting Tables \ref{tab:random} and \ref{tab:random2} (see Theorem \ref{t:random}). In every one of these cases, we can verify property \eqref{e:star} using {\sc Magma}. 

To do this, we first construct $G$ and $H$ using the functions 
\[
\mbox{\texttt{AutomorphismGroupSimpleGroup} and \texttt{MaximalSubgroups}.}
\]
Next we identify a set $R$ of $(H,H)$ double coset representatives and then for each $x \in R$ we seek an element $y \in G$ (by random search) such that $H \cap H^y = H^x \cap H^y=1$. Notice that \eqref{e:star} holds if and only if such an element $y$ exists for each $x \in R$. As demonstrated by the following example, it is easy to implement this approach in {\sc Magma}.

\begin{ex}
Suppose $G_0 = \O_{8}^{+}(2)$, $G = G_0.3$ and $H$ is of type ${\rm O}_{2}^{-}(2) \times \GU_3(2)$. Here $Q(G) = 2071/2800>1/2$ and so this is one of the cases we need to consider. We proceed as follows, noting that $G$ has a unique conjugacy class of maximal subgroups of order $11664$:

\vspace{2mm}

{\small
\begin{verbatim}
G:=AutomorphismGroupSimpleGroup("O+",8,2);
S:=LowIndexSubgroups(G,2);
G:=S[1];
M:=MaximalSubgroups(G:OrderEqual:=11664);
H:=M[1]`subgroup;
R,T:=DoubleCosetRepresentatives(G,H,H);
z:=0;
for x in R do 
  if exists(y){y : y in G | #(H meet H^y) eq 1 and #(H^x meet H^y) eq 1} 
    then z:=z+1; 
  end if;
end for;  
z eq #R;
\end{verbatim}}
\end{ex}

\vspace{1mm}

\noindent This returns \texttt{true} and we conclude that \eqref{e:star} holds. An entirely similar approach is effective for all of the relevant groups in Tables \ref{tab:random} and \ref{tab:random2}. 

\section{Proof of Theorems \ref{t:main11} and \ref{t:main12}}\label{s:main11}

Let $G \leqs {\rm Sym}(\O)$ be a permutation group in $\mathcal{G}$ with socle $G_0$ and point stabiliser $H$. Let $\omega(G)$ and $\a(G)$ denote the clique and independence numbers of $\Sigma(G)$, respectively.

\subsection{Clique number}\label{ss:clique} In view of Proposition \ref{p:bound} and Theorem \ref{t:random}, together with our work in Sections \ref{ss:psl2_c2} and \ref{ss:psl2_c3}, it remains to verify the bound $\omega(G) \geqs 5$ for the groups appearing in Tables \ref{tab:random} and \ref{tab:random2}. With the aid of {\sc Magma}, this is a straightforward exercise, working with a suitable permutation representation of $G$ and $H$. For example, the following function for checking the bound $\omega(G) \geqs 5$ can be used effectively in every case. This uses the fact that $\omega(G) \geqs 5$ if and only if there exist elements $x_2, \ldots, x_5$ in $G$ such that $H^{x_i} \cap H^{x_j} = 1$ for all distinct $i,j \in \{1, \ldots, 5\}$, where $x_1 = 1$.

\vspace{2mm}

{\small
\begin{verbatim}
clique:=function(G,H);
exists(x2){y : y in G | #(H meet H^y) eq 1};
exists(x3){y : y in G | #(H meet H^y) eq 1 and #(H^x2 meet H^y) eq 1};
exists(x4){y : y in G | #(H meet H^y) eq 1 and #(H^x2 meet H^y) eq 1 
  and #(H^x3 meet H^y) eq 1};
exists(x5){y : y in G | #(H meet H^y) eq 1 and #(H^x2 meet H^y) eq 1 
  and #(H^x3 meet H^y) eq 1 and #(H^x4 meet H^y) eq 1};
return "omega(G) is at least 5";
end function;
\end{verbatim}}

\vs

\subsection{Independence number}\label{ss:indep}

Now let us turn to the independence number $\a(G)$ of $\Sigma(G)$. Here we apply work of Magaard and Waldecker \cite{MagaardIndependence2,MagaardIndependence3} and we begin with the following trivial observation.

\begin{lem}\label{l:pt}
Let $G \leqs {\rm Sym}(\O)$ be a permutation group and let $c$ be a positive integer. Then $\a(G) \leqs c$ only if every $(c+1)$-point stabiliser in $G$ is trivial. 
\end{lem}

\begin{prop}\label{p:ag2}
Let $G \leqs {\rm Sym}(\O)$ be a group in $\mathcal{G}$ with socle $G_0$ and point stabiliser $H$. Then $\a(G) = 2$ if and only if $G = A_5$ and $\O$ is the set of $2$-element subsets of $\{1, \ldots, 5\}$.
\end{prop}

\begin{proof}
First assume $G = G_0$ is simple. We apply \cite[Theorem 3.20]{MagaardIndependence2}, which gives a complete list of all the finite simple transitive groups with the property that every $3$-point stabiliser is trivial. The theorem implies that $G = {\rm L}_2(q)$, ${}^2B_2(q)$ or ${\rm L}_{3}(4)$, and by inspection (using \cite[Theorem 2]{Burness2020base}) we deduce that $G = {\rm L}_{2}(4)$ with $H = D_6$ is the only example with $G$ primitive and $b(G)=2$. Here $\Sigma(G)$ is the Johnson graph $J(5,2)$ and $\a(G)=2$ (note that $G$ is permutation isomorphic to $A_5$ acting on the set of $2$-element subsets of $\{1, \ldots, 5\}$).

Now assume $G \ne G_0$. Then $G_0 \leqs {\rm Sym}(\O)$ is a transitive group and every $3$-point stabiliser is trivial, so as above the possibilities for $G_0$ and $\O$ are described in \cite[Theorem 3.20]{MagaardIndependence2} and by appealing to \cite[Theorem 1.3]{MagaardIndependence2} we see that $G_0 = {\rm L}_{2}(q)$ with $q=p^f$. More precisely, either $p$ is odd and $G = {\rm PGL}_{2}(q)$ or $G_0.2= \la G_0, \delta\phi^{f/2}\ra$ (in the notation of Section \ref{ss:psl2}), or $p=2$, $f$ is a prime and $G = \Aut(G_0)$. The cases with $q \in \{4,5\}$ can be checked using {\sc Magma} \cite{Magma} and we find that there are no groups $G\ne G_0$ with $\alpha(G) = 2$. Finally, let us assume $q\geqs 7$, in which case the relevant possibilities are labelled (a)--(d) in \cite[Theorem 3.20(1)]{MagaardIndependence2}. In (a) and (d), it is easy to check that $b(G_0)>2$, while the groups in (b) and (c) are imprimitive. Therefore, none of these cases arise and the proof is complete.
\end{proof}

\begin{prop}\label{p:ag3}
Let $G \leqs {\rm Sym}(\O)$ be a group in $\mathcal{G}$ with socle $G_0$ and point stabiliser $H$. Then $\a(G) \ne 3$.
\end{prop}

\begin{proof}
Seeking a contradiction, suppose $\a(G)=3$, in which case every $4$-point stabiliser is trivial. We may also assume that there exists a nontrivial $3$-point stabiliser. If $G=G_0$ then the possibilities for $G$ are described in \cite[Theorem 1.1(i)]{MagaardIndependence3}, but we find that there are no compatible examples with $G$ primitive, $H$ soluble and $b(G)=2$. The cases with $G \ne G_0$ are recorded in \cite[Theorem 1.2]{MagaardIndependence3} and once again we see that there are no valid examples.
\end{proof}

\vs

This completes the proof of Theorems \ref{t:main11} and \ref{t:main12}.

\section{Proof of Theorem \ref{t:main2}}\label{s:main2}

In this final section we prove Theorem \ref{t:main2} on the groups $G \in \mathcal{G}$ with a unique regular suborbit. If $G \in \mathcal{L}$ then we refer the reader to Propositions \ref{p:psl2_c2reg} and \ref{p:psl2_c3reg}, so we may assume $G \in \mathcal{G} \setminus \mathcal{L}$. If $Q(G) \geqs 1/4$ then we can simply read off the relevant groups in Tables \ref{tab:random} and \ref{tab:random2}, so we may assume $Q(G)<1/4$. Then in view of \eqref{e:QG}, it follows that $G$ has a unique regular suborbit only if 
\[
|H|^2 > \frac{3}{4}|G|,
\]
where $H$ is a point stabiliser. The following result reveals that there are no such groups and this completes the proof of Theorem \ref{t:main2}.

\begin{prop}\label{p:size}
Let $G$ be a group in $\mathcal{G}\setminus \mathcal{L}$ with point stabiliser $H$. If $Q(G)<1/4$ then $|H|^2 \leqs \frac{3}{4}|G|$.
\end{prop}

\begin{proof}
First assume $G_0 = A_m$ is an alternating group. If $m \leqs 12$ then it is easy to verify the desired bound with the aid of {\sc Magma}. On the other hand, if $m>12$ then by inspecting \cite[Table 14]{SolubleStabiliser} and \cite[Table 4]{Burness2020base} we deduce that $m$ is a prime, $H = {\rm AGL}_1(m)\cap G$ and we have 
\[
\frac{|H|^2}{|G|} \leqs \frac{m(m-1)}{(m-2)!}<\frac{3}{4}
\]
as required. Similarly, if $G_0$ is a sporadic group then the possibilities for $G$ and $H$ can be determined by combining the information in Table \ref{tab:random} and \cite[Table 4]{Burness2020base} with the tables of maximal subgroups in \cite{Wilson}. In each case, it is a straightforward exercise to show that $|H|^2 \leqs \frac{3}{4}|G|$.

Next suppose $G_0$ is an exceptional group of Lie type. Then either $H=N_G(T)$ for some maximal torus $T$ of $G_0$ (see \cite[Table 5.2]{MaxExceptional}), or $(G,H)$ is recorded in Table \ref{tab:ex}. One can now verify the bound $|H|^2 \leqs \frac{3}{4}|G|$ by inspection. For example, if $G_0 ={}^2B_2(q)$ with $q = 2^f$ and $f \geqs 3$ odd, then $|G| \geqs q^2(q^2+1)(q-1)$ and by inspecting \cite[Table 5.2]{MaxExceptional} we deduce that $|H|\leqs 4(q+\sqrt{2q}+1)\log q$. A routine calculation shows that $|H|^2 \leqs \frac{3}{4}|G|$. 

Finally, let us assume $G_0$ is a classical group. Then $(G,H)$ is one of the cases recorded in Table \ref{tab:cases} and once again the result follows by inspection (recall that in each case, the precise structure of $H$  is given in \cite{KleidmanLiebeckClassicalGroups}). For instance, if $G_0 = {\rm L}_4^{\e}(q)$ and $H$ is of type $\GL_1^{\e}(q)\wr S_4$ with $q \geqs 3$, then $|H| \leqs 48(q+1)^3\log q$ and the result follows since $|G|>\frac{1}{8}q^{15}$. In fact, one can check that the only case in Table \ref{tab:cases} with $|H|^2>\frac{3}{4}|G|$ is where $G = {\rm PGSp}_6(3)$ and $H$ is of type ${\rm Sp}_2(3) \wr S_3$. But $b(G)=3$ in this case (see \cite[Table 7]{Burness2020base}), so this is not a group in $\mathcal{G}$.
\end{proof}

\end{document}